\documentclass[11pt,a4paper]{amsart}
\pdfoutput=1
\usepackage[centering, margin=1in]{geometry}
\usepackage{amsmath,amsfonts,amsthm,amssymb,enumitem,mathrsfs,mathtools}
\usepackage[stretch=10]{microtype}
\usepackage[unicode]{hyperref}
\usepackage{xcolor}
\usepackage[normalem]{ulem}
\usepackage{cancel}

\definecolor{dark-red}{rgb}{0.4,0.15,0.15}
\definecolor{dark-blue}{rgb}{0.15,0.15,0.4}
\definecolor{medium-blue}{rgb}{0,0,0.5}
\hypersetup{
	pdftitle={The Second Moment of Rankin--Selberg \texorpdfstring{$L$}{L}-Functions in Conductor-Dropping Regimes},
	pdfauthor={Peter Humphries and Rizwanur Khan},
	pdfnewwindow=true,
	colorlinks, linkcolor={dark-red},
	citecolor={dark-blue}, urlcolor={medium-blue}
}
\allowdisplaybreaks

\newcommand{\BB}{\mathcal{B}}
\newcommand{\C}{\mathbb{C}}

\newcommand{\e}{\varepsilon}
\newcommand{\EE}{\mathcal{E}}
\newcommand{\Eis}{\mathrm{Eis}}

\newcommand{\Hb}{\mathbb{H}}

\newcommand{\hol}{\mathrm{hol}}

\newcommand{\MM}{\mathcal{M}}
\newcommand{\Maass}{\mathrm{Maass}}
\newcommand{\N}{\mathbb{N}}

\newcommand{\R}{\mathbb{R}}
\newcommand{\reg}{\mathrm{reg}}

\newcommand{\Z}{\mathbb{Z}}

\DeclareMathOperator{\ad}{ad}

\DeclareMathOperator{\GL}{GL}

\DeclareMathOperator*{\Res}{Res}

\DeclareMathOperator{\sgn}{sgn}
\DeclareMathOperator{\SL}{SL}
\DeclareMathOperator{\vol}{vol}
\numberwithin{equation}{section}
\newtheorem{theorem}[equation]{Theorem}

\newtheorem{corollary}[equation]{Corollary}

\newtheorem{lemma}[equation]{Lemma}
\newtheorem{proposition}[equation]{Proposition}
\theoremstyle{remark}
\newtheorem{remark}[equation]{Remark}
\theoremstyle{definition}

\begin{document}

\title{The Second Moment of Rankin--Selberg $L$-Functions in Conductor-Dropping Regimes}

\author{Peter Humphries}

\address{Department of Mathematics, University of Virginia, Charlottesville, VA 22904, USA}

\email{\href{mailto:pclhumphries@gmail.com}{pclhumphries@gmail.com}}

\urladdr{\href{https://sites.google.com/view/peterhumphries/}{https://sites.google.com/view/peterhumphries/}}

\author{Rizwanur Khan}

\address{Department of Mathematical Sciences, University of Texas at Dallas, Richardson, TX 75080, USA}

\email{\href{mailto:rizwanur.khan@utdallas.edu}{rizwanur.khan@utdallas.edu}}

\urladdr{\href{https://profiles.utdallas.edu/rizwanur.khan}{https://profiles.utdallas.edu/rizwanur.khan}}


\thanks{The first author was supported by the National Science Foundation grant DMS-2302079 and the Simons Foundation (award 965056). The second author was supported by the National Science Foundation grants DMS-2344044 and DMS-2341239.}

\begin{abstract}
We prove an asymptotic formula for the second moment of $L$-functions associated to the Rankin--Selberg convolution of two holomorphic Hecke cusp forms with equal weight.
\end{abstract}

\maketitle

\section{Introduction}

\subsection{Results}

One of the most outstanding open cases of the subconvexity problem for $L$-functions is that of the adjoint (or symmetric square) $L$-function $L(s,\ad f)$, where $f$ is a holomorphic Hecke cusp form of even weight $k_f$ for $\SL_2(\Z)$. At the central point, the convexity bound is $L(\frac{1}{2},\ad f) \ll k_f^{1/2}$. The goal for the subconvexity problem is to save any power of $k_f$ over this bound, while the best bound currently known is Soundararajan's so-called weak subconvex bound that saves almost a whole power of $\log k_f$ \cite{Sou10}. A possible approach to this problem involves the second moment
\begin{equation}
\label{eqn:diff-family}
\sum_{\substack{f \in \BB_{\hol} \\ k_f = k}} \frac{L\left(\frac{1}{2},\ad f\right)^2}{L(1,\ad f)}
\end{equation}
where $\BB_{\hol}$ denotes an orthonormal basis of holomorphic Hecke cusp forms $f$ of weight $k_f$. If an asymptotic formula (with a main term of size $k \log^3 k$ according to the random matrix theory recipe \cite{CFKRS05}) can be obtained with an error term that saves a power of $k$ over the main term, then one can likely insert an amplifier to obtain the desired subconvex bound.

Although the second moment \eqref{eqn:diff-family} has the standard $\log$(conductor):$\log$(family size) ratio of $4$, since we sum over a family of about $\frac{k}{12}$ forms $f$ and $L(\frac{1}{2},\ad f)^2$ has analytic conductor $k^4$, an asymptotic formula for this moment unfortunately remains elusive. One might hope to prove at the very least the upper bound $O_{\e}(k^{1 + \e})$ for this moment, which would recover the convexity bound, but even this is far from known (probably the upper bound $O_{\e}(k^{\frac{6}{5} + \e})$ can be obtained following the same method used for \cite[Theorem 1.3]{KY23}). The situation is a bit better in the \emph{level} aspect, where Iwaniec and Michel \cite{IM01} have proven an upper bound that is sharp up to an $\e$-power of the level.

In this paper, we consider a different but related family. We investigate the second moment of Rankin--Selberg $L$-functions
\begin{equation}
\label{eqn:our-family}
\sum_{\substack{f \in \BB_{\hol} \\ k_f = k}} \frac{L\left(\frac{1}{2},f \otimes g\right)^2}{L(1,\ad f)},
\end{equation}
where $g$ is also a holomorphic Hecke cusp form with weight $k_g = k$. Here the Rankin--Selberg $L$-function is defined for $\Re(s) > 1$ by
\begin{equation}
\label{eqn:RSdefeq}
L(s,f \otimes g) \coloneqq \zeta(2s) \sum_{n = 1}^{\infty} \frac{\lambda_f(n) \lambda_g(n)}{n^s},
\end{equation}
where $\lambda_f(n)$ and $\lambda_g(n)$ denote the Hecke eigenvalues of $f$ and $g$, and then by meromorphic continuation to all of $\C$ with at most a simple pole at $s = 1$, which only occurs when $f = g$. The analytic conductor of $L(\frac{1}{2},f \otimes g)$ is of size $\max\{k_f^2,k_g^2\} (|k_f - k_g| + 1)^2$, which drops to $k^2$ when $k_f = k_g = k$. As before, an asymptotic formula with a power-saving error term would be significant because it would likely facilitate subconvex bounds through amplification. These subconvex bounds would be for $L(\frac{1}{2}, f \otimes g)$, which would also include the value $L(\frac{1}{2}, f \otimes f) = \zeta(\frac{1}{2}) L(\frac{1}{2}, \ad f).$ Our main result \emph{is} an asymptotic formula for \eqref{eqn:our-family}, albeit not with a power-saving error term.

\begin{theorem}
\label{thm:mainthm}
Let $\BB_{\hol}$ denote an orthonormal basis of holomorphic Hecke cusp forms $f$ of positive even weight $k_f \in 2\N$. Let $g \in \BB_{\hol}$ be a holomorphic Hecke cusp form of positive even weight $k \in 2\N$. For $j \in \{0,1,2,3\}$, there exist polynomials $P_j$ of degree $j$ such that
\begin{equation}
\label{eqn:mainthmasymp}
\sum_{\substack{f \in \BB_{\hol} \\ k_f = k}} \frac{L\left(\frac{1}{2},f \otimes g\right)^2}{L(1,\ad f)} = k \sum_{j = 0}^{3} P_{j}(\log k) L^{(3 - j)}(1,\ad g) + O_{\e}\left(k (\log k)^{-\frac{1}{2} + \e}\right)
\end{equation}
for any $\e > 0$. Moreover, the main term $k \sum_{j = 0}^{3} P_{j}(\log k) L^{(3 - j)}(1,\ad g)$ is positive for $k$ sufficiently large and satisfies the bounds
\begin{equation}
\label{eqn:maintermbounds}
k (\log k)^2 \ll k \sum_{j = 0}^{3} P_{j}(\log k) L^{(3 - j)}(1,\ad g) \ll k (\log k)^6.
\end{equation}
\end{theorem}

In contrast to the moment \eqref{eqn:diff-family}, the upper bound $O_{\e}(k^{1 + \e})$ for \eqref{eqn:our-family} is straightforward to prove using an approximate functional equation and a large sieve inequality \cite{Lam14}. Given this favourable starting point coupled with the fact that the $\log$(conductor):$\log$(family size) ratio is $4$ for this moment, one may expect to refine this upper bound to an asymptotic formula with relative ease. However, the usual heuristics can be misleading in conductor-dropping settings, as we saw for \eqref{eqn:diff-family}. In fact, our proof shows that an asymptotic for \eqref{eqn:our-family} with a power-saving error seems out of reach unless one already has available subconvex bounds for $L(\frac{1}{2}, \ad g)$ and its $\GL_2$ twists. With this looming barrier, the best that we can hope for is a logarithmic savings in the error term. This is what we obtain via a period integral approach coupled with Soundararajan's weak subconvex bounds.

On top of these issues, it is a delicate process to show that we actually have an asymptotic formula at hand. This is because the holomorphic Hecke cusp form $g$ is not fixed and we have relatively poor knowledge on lower and upper bounds for $L(1,\ad g)$ and its derivatives. We have (cf.\ {\cite[Proposition 3.2]{LW06}, \cite{GHL94}}) that 
\begin{equation}
\label{eqn:1-bounds}
L(1,\ad g) \gg \frac{1}{\log k}, \quad L^{(j)}(1,\ad g) \ll (\log k)^{3+j}
\end{equation}
for $j \in \{0,1,2,3\}$, which is insufficient to show that the main term is larger than the error term in \eqref{eqn:mainthmasymp}. We manage to circumvent this difficultly by exploiting some hidden structure in the main term that is not apparent when it is worked out as in \eqref{eqn:mainthmasymp}. Note that on the assumption of the generalised Riemann hypothesis, we would have bounds in \eqref{eqn:1-bounds} on the scale of powers of $\log \log k$ rather than powers of $\log k$ (cf.\ \cite[Theorem 3 (ii)]{LW06}), and so there would be no such difficulties. 

Besides the relation to subconvexity problems, the second moment \eqref{eqn:our-family} is interesting in its own right. Upper bounds for such moments in conductor-dropping regimes, and in greater generality, have been considered by Blomer, who provided in \cite[Theorem 1]{Blo12} a different proof (without recourse to the large sieve) of the upper bound $O_{\e}(k^{1 + \e})$. He showed that \eqref{eqn:our-family} is bounded from above by a constant multiple of 
\begin{equation}
\label{eqn:Blomerint}
k \int_{\SL_2(\Z) \backslash \Hb} \Im(z)^k |g(z)|^2 \left|E^{\ast}\left(z,\frac{1}{2}\right)\right|^2 \, d\mu(z),
\end{equation}
where $d\mu(z) = y^{-2} \, dx \, dy$ denotes the $\SL_2(\R)$-invariant measure on the modular surface $\SL_2(\Z) \backslash \Hb$ and $E^{\ast}(z,s)$ denotes the completed real-analytic Eisenstein series (see \eqref{eqn:Eunfold} and \eqref{eqn:Eastdefeq} below). Blomer then bounded $E^{\ast}(z,\frac{1}{2})$ from above by the real-analytic Eisenstein series $E(z,1 + \e)$ and finally used the Rankin--Selberg method to obtain the upper bound $O_{\e}(k^{1 + \e})$.

We improve upon \cite[Theorem 1]{Blo12} by following a similar initial approach as Blomer and similarly arrive at the integral \eqref{eqn:Blomerint}. At this point, our strategy deviates from Blomer's strategy. We instead view this integral as the Petersson inner product of $\Im(z)^k |g(z)|^2$ and $|E^{\ast}(z,\frac{1}{2})|^2$ and apply Parseval's identity. The ensuing spectral expansion gives rise to the main term in \eqref{eqn:mainthmasymp} together with a spectral sum of triple products of automorphic forms. We use the Watson--Ichino triple product formula to express the resulting triple products of automorphic forms in terms of $L$-functions, and finally we input pointwise bounds for these $L$-functions to show that this spectral sum gives rise to the error term in \eqref{eqn:mainthmasymp}.

Finally, we comment on the hidden structure of the main term that was alluded to above. If we were working with approximate functional equations, then the diagonal contribution to the second moment would essentially be
\begin{equation}
\label{eqn:diagonal}
k \sum_{n = 1}^{\infty} \frac{\lambda_g(n)^2}{n}V_k(n)^2,
\end{equation}
where $\lambda_g(n)$ denotes the $n$-th Hecke eigenvalue of $g$ and $V_k(n)$ is a weight function that arises in the approximate functional equation. Exploiting the positivity structure of this expression, we have that \eqref{eqn:diagonal} is bounded from below by the $n=1$ term
\[k V_k(1)^2,\]
and we can easily compute that this is $\gg k \log^2 k$. However, executing this idea is complicated by the fact that we do not actually work with approximate functional equations, and so we must first manipulate our main term into a form resembling \eqref{eqn:diagonal}. Moreover, there arises an additional off-diagonal main term that is a negative multiple of $k L(1,\ad g)$. The upper bound in \eqref{eqn:1-bounds} is insufficient to rule out the possibility that this off-diagonal term cancels out the diagonal term. To get around this, we come up with another trick in which we express $k L(1,\ad g)$ as an integral that can be compared directly with \eqref{eqn:diagonal} term by term.

\subsection{Generalisations}

The methods in this paper can also be used to deduce the asymptotic behaviour of similar moments in either the level aspect or spectral aspect. For example, if we fix the weight $k$ but now allow $f$ to vary over an orthonormal basis $\BB_{\hol}^{\ast}(q)$ of holomorphic Hecke newforms of increasing prime level $q$, and if $g$ is also a holomorphic Hecke newform of prime level $q$, then we may show that there exist polynomials $Q_j$ of degree $j$, dependent on $k$, such that
\[\sum_{\substack{f \in \BB_{\hol}^{\ast}(q) \\ k_f = k}} \frac{L\left(\frac{1}{2},f \otimes g\right)^2}{L(1,\ad f)} = q \sum_{j = 0}^{3} Q_{j}(\log q) L^{(3 - j)}(1,\ad g) + O_{k,\e}\left(q (\log q)^{-\frac{1}{2} + \e}\right),\]
where the main term is positive for $q$ sufficiently large and satisfies the bounds
\[q (\log q)^2 \ll_k q \sum_{j = 0}^{3} Q_{j}(\log q) L^{(3 - j)}(1,\ad g) \ll_k q (\log q)^6.\]

Alternatively, we may replace the average over holomorphic Hecke cusp forms $f \in \BB_{\hol}$ of weight $k$ with an average over Hecke--Maa\ss{} cusp forms $f \in \BB_0$ of spectral parameter $t_f \in [T,T + 1]$, and let $g$ be a Hecke--Maa\ss{} cusp form of spectral parameter $t_g \in [T,T + 1]$. We may then show that there exist polynomials $R_j$ of degree $j$ such that
\[\sum_{\substack{f \in \BB_0 \\ t_f \in [T,T + 1]}} \frac{L\left(\frac{1}{2},f \otimes g\right)^2}{L(1,\ad f)} = T \sum_{j = 0}^{3} R_{j}(\log T) L^{(3 - j)}(1,\ad g) + O_{\e}\left(T (\log T)^{-\frac{1}{2} + \e}\right),\]
where the main term is positive for $T$ sufficiently large and satisfies the bounds
\[T (\log T)^2 \ll T \sum_{j = 0}^{3} R_{j}(\log T) L^{(3 - j)}(1,\ad g) \ll T (\log T)^6.\]

Moreover, in all of the cases above, we may prove asymptotic formul\ae{} for the second moment of Rankin--Selberg $L$-functions not necessarily at the central point but instead at the point $\frac{1}{2} + it$ for some fixed $t \in \R$.

Finally, the methods in this paper may also be applied to the first moment of triple product $L$-functions at the central point
\[\sum_{\substack{f \in \BB_{\hol} \\ k_f = k}} \frac{L\left(\frac{1}{2},f \otimes g \otimes h\right)}{L(1,\ad f)}.\]
Here $g$ is again a holomorphic Hecke cusp form of weight $k$, while $h$ is either a \emph{fixed} holomorphic Hecke cusp form of \emph{fixed} weight $k_h$ or a \emph{fixed} Hecke--Maa\ss{} cusp form of \emph{fixed} spectral parameter $t_h$. However, while we would be able to extract the expected main term $ c_h L(1, \ad g) k$ for this moment, for some constant $c_h$ depending on $h$, and obtain an error term that saves a fractional power of $\log k$,  we would not be able to prove that the main term dominates the error term, as the best available lower bound $L(1, \ad g)\gg (\log k)^{-1}$ is inadequate.

All of these generalisations have similar counterparts in the work of Blomer \cite{Blo12}, who observes that one can prove upper bounds rather than asymptotic formul\ae{} using his approach. Blomer also uses his approach to give upper bounds for the second moment of Rankin--Selberg $L$-functions in conductor-dropping regimes for automorphic forms on $\GL_n$ with $n \geq 3$ rather than just for automorphic forms on $\GL_2$. In this more general setting, however, our method is no longer applicable, since we rely upon the Watson--Ichino triple product formula to express integrals of three automorphic forms in terms of $L$-functions, and such a formula does not exist for automorphic forms on $\GL_n$ with $n \geq 3$.

\section{An Exact Identity for Moments of \texorpdfstring{$L$}{L}-Functions}

The strategy towards proving \hyperref[thm:mainthm]{Theorem \ref*{thm:mainthm}} is to first prove an exact identity for a moment of $L$-functions that involves the moment of interest \eqref{eqn:our-family}. We subsequently extract this moment of interest and the main term $k \sum_{j = 0}^{3} P_{j}(\log k) L^{(3 - j)}(1,\ad g)$ and then prove that all of the remaining terms contribute a smaller error term. The exact identity takes the following form.

\begin{theorem}
\label{thm:exactidentity}
Let $\BB_{\hol}$ denote an orthonormal basis of holomorphic Hecke cusp forms $f$ of positive even weight $k_f \in 2\N$ and let $\BB_0$ denotes an orthonormal basis of Hecke--Maa\ss{} cusp forms $f$ of spectral parameter $t_f \in \R$ and parity $\epsilon_f \in \{1,-1\}$. Let $g \in \BB_{\hol}$ be a holomorphic Hecke cusp form of positive even weight $k \in 2\N$. We have the identity
\begin{equation}
\label{eqn:exactidentity}
\MM_{\hol} + \MM_{\Maass} + \MM_{\Eis} = \widetilde{\MM}_0 + \widetilde{\MM}_{\Maass} + \widetilde{\MM}_{\Eis},
\end{equation}
where
\begin{align}
\label{eqn:MMholdefeq}
\MM_{\hol} & \coloneqq \sum_{f \in \BB_{\hol}} \frac{L\left(\frac{1}{2},f \otimes g\right)^2}{L(1,\ad f)} h^{\hol}(k_f,k),	\\
\label{eqn:MMMaassdefeq}
\MM_{\Maass} & \coloneqq \sum_{f \in \BB_0} \frac{L\left(\frac{1}{2},f \otimes g\right)^2}{L(1,\ad f)} h(t_f,k),	\\
\label{eqn:MMEisdefeq}
\MM_{\Eis} & \coloneqq \frac{1}{2\pi} \int_{-\infty}^{\infty} \frac{L\left(\frac{1}{2} + it,g\right)^2 L\left(\frac{1}{2} - it,g\right)^2}{\zeta(1 + 2it) \zeta(1 - 2it)} h(t,k) \, dt,	\\
\label{eqn:tildeMM0defeq}
\widetilde{\MM}_0 & \coloneqq 2 \frac{\Gamma(k)}{\Gamma\left(k - \frac{1}{2}\right)^2} \lim_{(w_1,w_2) \to (0,0)} \sum_{\pm_1, \pm_2} (2\pi)^{-2w_1 - 2w_2} \Gamma\left(\frac{1}{2} + w_1\right) \Gamma\left(\frac{1}{2} + w_2\right)	\\
\notag
& \qquad \times \Gamma(w_1 + w_2 + k) \frac{\zeta(1 + 2w_1) \zeta(1 + 2w_2) \zeta(1 + w_1 + w_2) L(1 + w_1 + w_2,\ad g)}{\zeta(2 + 2w_1 + 2w_2)},	\\
\label{eqn:tildeMMMaassdefeq}
\widetilde{\MM}_{\Maass} & \coloneqq \sum_{\substack{f \in \BB_0 \\ \epsilon_f = 1}} \theta_{f,g} \frac{\sqrt{L\left(\frac{1}{2},\ad g \otimes f\right) L\left(\frac{1}{2},f\right)^5}}{L(1,\ad f)} \widetilde{h}(t_f,k),	\\
\label{eqn:tildeMMEisdefeq}
\widetilde{\MM}_{\Eis} & \coloneqq \frac{1}{2\pi} \int_{-\infty}^{\infty} \frac{L\left(\frac{1}{2} + it,\ad g\right) \zeta\left(\frac{1}{2} + it\right)^3 \zeta\left(\frac{1}{2} - it\right)^2}{\zeta(1 + 2it) \zeta(1 - 2it)} \widetilde{h}(t,k) \, dt.
\end{align}
Here $\theta_{f,g}$ is a complex constant of absolute value $1$ dependent only on $f$ and $g$, while for $\ell \in 2\N$ and $t \in \R$, the spectral weights $h^{\hol}(\ell,k)$, $h(t,k)$, and $\widetilde{h}(t,k)$ are given by
\begin{align}
\label{eqn:hholellkdefeq}
h^{\hol}(\ell,k) & \coloneqq \begin{dcases*}
\frac{1}{\pi} \frac{\Gamma(k)}{\Gamma\left(k - \frac{1}{2}\right)^2} \frac{\Gamma\left(\frac{k + \ell - 1}{2}\right)^2 \Gamma\left(\frac{k - \ell + 1}{2}\right)^2}{\Gamma\left(\frac{k + \ell}{2}\right) \Gamma\left(\frac{k - \ell}{2} + 1\right)} & if $\ell \leq k$,	\\
0 & if $\ell > k$,
\end{dcases*}	\\
\label{eqn:htkdefeq}
h(t,k) & \coloneqq \frac{1}{\pi} \frac{\Gamma(k)}{\Gamma\left(k - \frac{1}{2}\right)^2} \frac{\Gamma\left(\frac{k}{2} + it\right)^2 \Gamma\left(\frac{k}{2} - it\right)^2}{\Gamma\left(\frac{k + 1}{2} + it\right) \Gamma\left(\frac{k + 1}{2} - it\right)},	\\
\label{eqn:tildehtkdefeq}
\widetilde{h}(t,k) & \coloneqq 2^{-1 - 2it} \pi^{-2 - it} \frac{\Gamma(k)}{\Gamma\left(k - \frac{1}{2}\right)^2} \frac{\Gamma\left(k - \frac{1}{2} + it\right) \Gamma\left(\frac{1}{4} + \frac{it}{2}\right)^2 \Gamma\left(\frac{1}{4} - \frac{it}{2}\right)^2}{\Gamma\left(\frac{1}{2} - it\right)}.
\end{align}
\end{theorem}

The proof of \hyperref[thm:exactidentity]{Theorem \ref*{thm:exactidentity}} involves spectrally expanding in two different ways an integral of four automorphic forms. Given a holomorphic Hecke cusp form $g \in \BB_{\hol}$ of positive even weight $k$, we let $G(z) \coloneqq \Im(z)^{k/2} g(z)$. We normalise $g$ such that $\langle G,G\rangle = 1$, where the weight $k$ Petersson inner product is
\[\langle F,G\rangle \coloneqq \int_{\SL_2(\Z) \backslash \Hb} F(z) \, \overline{G(z)} \, d\mu(z)\]
with the $\SL_2(\R)$-invariant measure $d\mu(z)$ given by $y^{-2} \, dx \, dy$, so that $\vol(\SL_2(\Z) \backslash \Hb) = \pi/3$. The holomorphic Hecke cusp form $g$ has the Fourier expansion
\begin{equation}
\label{eqn:gFourier}
g(z) = \sum_{n = 1}^{\infty} \rho_g(n) (4\pi n)^{\frac{k}{2}} e(nz),
\end{equation}
where the Fourier coefficients $\rho_g(n)$ satisfy
\begin{equation}
\label{eqn:rhog(n)}
\rho_g(n) = \rho_g(1) \frac{\lambda_g(n)}{\sqrt{n}}, \qquad |\rho_g(1)|^2 = \frac{\pi}{2\Gamma(k) L(1,\ad g)}.
\end{equation}
Here $\lambda_g(n)$ denotes the $n$-th Hecke eigenvalue of $g$. The Hecke eigenvalues are multiplicative: for $m,n \in \N$, they satisfy
\[\lambda_g(mn) = \sum_{d \mid (m,n)} \mu(d) \lambda_g\left(\frac{m}{d}\right) \lambda_g\left(\frac{n}{d}\right).\]
Moreover, since $g$ has level $1$, and hence is self-dual, the Hecke eigenvalues $\lambda_g(n)$ are real-valued, and we may normalise $g$ such that $\rho_g(1)$ is positive, so that the Fourier coefficients $\rho_g(n)$ are also real-valued.

For $s \in \C$, we let $E(z,s)$ denote the real analytic Eisenstein series, which is given for $\Re(s) > 1$ by
\begin{equation}
\label{eqn:Eunfold}
E(z,s) \coloneqq \sum_{\gamma \in \Gamma_{\infty} \backslash \SL_2(\Z)} \Im(\gamma z)^s,
\end{equation}
where $\Gamma_{\infty} \coloneqq \{ \pm \begin{psmallmatrix}
1 & n \\ 0 & 1 \end{psmallmatrix} : n \in \Z\}$. The completed Eisenstein series
\begin{equation}
\label{eqn:Eastdefeq}
E^{\ast}(z,s) \coloneqq \pi^{-s} \Gamma(s) \zeta(s) E(z,s)
\end{equation}
extends meromorphically to the entire complex plane with poles only at $s = 1$ and $s = 0$.

For $\Re(s_1),\Re(s_2) > 1$, we consider the integral
\begin{equation}
\label{eqn:mainint}
\int_{\SL_2(\Z) \backslash \Hb} |G(z)|^2 E^{\ast}(z,s_1) E^{\ast}(z,s_2) \, d\mu(z).
\end{equation}
This integral converges absolutely as $|G(z)|^2$ is rapidly decreasing at the cusp of $\SL_2(\Z) \backslash \Hb$, whereas $E^{\ast}(z,s_1) E^{\ast}(z,s_2)$ is of moderate growth.

On the one hand, the integral \eqref{eqn:mainint} is equal to the inner product $\langle G E^{\ast}(\cdot,s_1),G \overline{E^{\ast}(\cdot,s_2)}\rangle$, where $G E^{\ast}(\cdot,s_1)$ and $G \overline{E^{\ast}(\cdot,s_2)}$ are rapidly decaying automorphic functions of weight $k$. We may therefore apply Parseval's identity in order to spectrally expand this inner product in terms of automorphic forms of weight $k$. Using the Rankin--Selberg unfolding method, we then show in \hyperref[prop:firstexpansion]{Proposition \ref*{prop:firstexpansion}} that, once appropriately normalised and holomorphically extended to $(s_1,s_2) = (\frac{1}{2},\frac{1}{2})$, this inner product is equal to $\MM_{\hol} + \MM_{\Maass} + \MM_{\Eis}$.

On the other hand, the integral \eqref{eqn:mainint} is equal to the inner product $\langle E^{\ast}(\cdot,s_1) E^{\ast}(\cdot,s_2), |G|^2\rangle$, where $E^{\ast}(\cdot,s_1) E^{\ast}(\cdot,s_2)$ is an automorphic function of moderate growth of weight $0$, while $|G|^2$ is a rapidly decaying automorphic function of weight $0$. We may therefore apply a regularised version of Parseval's identity for functions of moderate growth in order to spectrally expand this inner product in terms of automorphic forms of weight $0$. Using the Watson--Ichino triple product formula, we then show in \hyperref[prop:secondexpansion]{Proposition \ref*{prop:secondexpansion}} that, once appropriately normalised and holomorphically extended to $(s_1,s_2) = (\frac{1}{2},\frac{1}{2})$, this inner product is equal to $\widetilde{\MM}_0 + \widetilde{\MM}_{\Maass} + \widetilde{\MM}_{\Eis}$. In this way, we deduce the identity \eqref{eqn:exactidentity}.

\subsection{The First Spectral Expansion}

We first consider the spectral expansion of the inner product $\langle G E^{\ast}(\cdot,s_1),G \overline{E^{\ast}(\cdot,s_2)}\rangle$.

\begin{proposition}
\label{prop:firstexpansion}
Let $g$ be a holomorphic Hecke cusp form of positive even weight $k$ and let $G(z) \coloneqq \Im(z)^{k/2} g(z)$. The normalised inner product
\begin{equation}
\label{eqn:firstnormalisedinnerproduct}
\pi^{-3 - s_1 - s_2} \frac{\Gamma(k)^2}{\Gamma\left(k - \frac{1}{2}\right)^2} \Gamma(s_1) \Gamma(s_2) \zeta(2s_1) \zeta(2s_2) L(1,\ad g) \langle G E(\cdot,s_1),G \overline{E(\cdot,s_2)}\rangle
\end{equation}
is absolutely convergent for $\Re(s_1),\Re(s_2) > 1$ and extends holomorphically to all of $\C^2$. At $(s_1,s_2) = (\frac{1}{2},\frac{1}{2})$, it is equal to
\[\MM_{\hol} + \MM_{\Maass} + \MM_{\Eis},\]
with $\MM_{\hol}$, $\MM_{\Maass}$, and $\MM_{\Eis}$ as in \eqref{eqn:MMholdefeq}, \eqref{eqn:MMMaassdefeq}, and \eqref{eqn:MMEisdefeq} respectively.
\end{proposition}

The first step of the proof of \hyperref[prop:firstexpansion]{Proposition \ref*{prop:firstexpansion}} is to apply Parseval's identity for square-integrable automorphic forms of weight $k$. Following \cite[Sections 3B and 4A]{BHMWW24}, we recall the following classification of the automorphic forms of weight $k$ that appear in this spectral expansion.
\begin{itemize}
\item First, for each positive even integer $\ell$ lesser than or equal to $k$, associated to each holomorphic Hecke cusp form $f \in \BB_{\hol}$ of weight $k_f = \ell$ is an $L^2$-normalised shifted holomorphic Hecke cusp form $F_k$ of weight $k$, which is obtained by repeatedly applying raising operators $\frac{k - \ell}{2}$ times to $\Im(z)^{\ell/2} f(z)$.
\item Similarly, associated to each Hecke--Maa\ss{} form $f \in \BB_0$ is an $L^2$-normalised shifted Hecke--Maa\ss{} cusp form $F_k$ of weight $k$, which is obtained by applying raising operators $\frac{k}{2}$ times to $f(z)$.
\item Finally, let $E(z,\frac{1}{2} + it)$ denote the real analytic Eisenstein series of spectral parameter $t$. Associated to this Eisenstein series is a normalised shifted Eisenstein series $E_k(z,\frac{1}{2} + it)$ of weight $k$, which is once more obtained by applying raising operators $\frac{k}{2}$ times to $E(z,\frac{1}{2} + it)$.
\end{itemize}

\begin{lemma}
\label{lem:firstexpansionParseval}
For $\Re(s_1),\Re(s_2) > 1$,
\begin{multline}
\label{eqn:GEGE}
\langle G E(\cdot,s_1),G \overline{E(\cdot,s_2)}\rangle = \sum_{\substack{\ell = 2 \\ \ell \equiv 0 \hspace{-.25cm} \pmod{2}}}^{k} \sum_{\substack{f \in \BB_{\hol} \\ k_f = \ell}} \langle G E(\cdot,s_1),F_k\rangle \langle F_k, G \overline{E(\cdot,s_2)}\rangle	\\
+ \sum_{f \in \BB_0} \langle G E(\cdot,s_1),F_k\rangle \langle F_k, G \overline{E(\cdot,s_2)}\rangle	\\
+ \frac{1}{4\pi} \int_{-\infty}^{\infty} \left\langle G E(\cdot,s_1),E_k\left(\cdot,\frac{1}{2} + it\right)\right\rangle \left\langle E_k\left(\cdot,\frac{1}{2} + it\right), G \overline{E(\cdot,s_2)}\right\rangle \, dt.
\end{multline}
\end{lemma}

\begin{proof}
This is a consequence of Parseval's identity for square-integrable automorphic forms of weight $k$ \cite[Section 3C]{BHMWW24}, noting that $G E(\cdot,s)$ is square-integrable as $G$ is rapidly decreasing and $E(\cdot,s)$ is of moderate growth.
\end{proof}

It remains to determine the various inner products appearing in the spectral expansion \eqref{eqn:GEGE}. These can be explicitly expressed in terms of ratios of $L$-functions and gamma functions via the Rankin--Selberg method.

\begin{lemma}
\label{lem:firstexpansionhol}
Let $g$ be a holomorphic Hecke cusp form of weight $k$, and let $F_k$ be an $L^2$-normalised shifted holomorphic Hecke cusp form $F_k$ of weight $k$ arising from an unshifted holomorphic Hecke cusp form $f$ of weight $\ell$, where $\ell$ is a positive even integer less than $k$. Then for $\Re(s) \geq \frac{1}{2}$ with $s \neq 1$,
\begin{multline*}
\langle F_k, G \overline{E(\cdot,s)}\rangle = \frac{(-1)^{\frac{k - \ell}{2}} 2^{1 - 2s} \pi^{2 - s}}{\sqrt{\Gamma(k) \Gamma\left(\frac{k + \ell}{2}\right) \Gamma\left(\frac{k - \ell}{2} + 1\right)}} \frac{\Gamma\left(s + \frac{k + \ell}{2} - 1\right) \Gamma\left(s + \frac{k - \ell}{2}\right)}{\Gamma(s)}	\\
\times \frac{L(s,f \otimes g)}{\zeta(2s) \sqrt{L(1,\ad f) L(1,\ad g)}}.
\end{multline*}
\end{lemma}

\begin{proof}
The shifted cusp form $F_k$ has the Fourier expansion
\begin{equation}
\label{eqn:FkholFourier}
F_k(z) = \sum_{n = 1}^{\infty} C_{k,\ell} \rho_f(n) W_{\frac{k}{2},\frac{\ell - 1}{2}}(4\pi ny) e(nx),
\end{equation}
where $W_{\alpha,\beta}(y)$ denotes the Whittaker function, the normalising constant $C_{k,\ell}$ is given by
\begin{equation}
\label{eqn:Ckell}
C_{k,\ell} \coloneqq (-1)^{\frac{k - \ell}{2}} \sqrt{\frac{\Gamma(\ell)}{\Gamma\left(\frac{k + \ell}{2}\right) \Gamma\left(\frac{k - \ell}{2} + 1\right)}},
\end{equation}
and the Fourier coefficients $\rho_f(n)$ satisfy \eqref{eqn:rhog(n)} \cite[Section 4A]{BHMWW24}. By unfolding the inner product $\langle F_k, G \overline{E(\cdot,s)}\rangle$ via \eqref{eqn:Eunfold}, inserting the Fourier expansions of $F_k$ and $g$ given by \eqref{eqn:FkholFourier} and \eqref{eqn:gFourier}, and using the identities \eqref{eqn:rhog(n)} for the Fourier coefficients, we deduce that for $\Re(s) > 1$,
\begin{align*}
\langle F_k, G \overline{E(\cdot,s)}\rangle & = \int_{0}^{\infty} y^{s + \frac{k}{2}} \int_{0}^{1} F_k(x + iy) \overline{g(x + iy)} \, dx \, \frac{dy}{y^2}	\\
& = C_{k,\ell} \sum_{n = 1}^{\infty} \rho_f(n) \rho_g(n) (4\pi n)^{k/2} \int_{0}^{\infty} W_{\frac{k}{2},\frac{\ell - 1}{2}}(4\pi ny) y^{s + \frac{k}{2} - 1} e^{-2\pi ny} \, \frac{dy}{y}	\\
& = C_{k,\ell} (4\pi)^{1 - s} \rho_f(1) \rho_g(1) \sum_{n = 1}^{\infty} \frac{\lambda_f(n) \lambda_g(n)}{n^s} \int_{0}^{\infty} W_{\frac{k}{2},\frac{\ell - 1}{2}}(y) y^{s + \frac{k}{2} - 1} e^{-\frac{y}{2}} \, \frac{dy}{y}	\\
& = \frac{2^{1 - 2s} \pi^{2 - s} C_{k,\ell}}{\sqrt{\Gamma(k) \Gamma(\ell)}} \frac{L(s,f \otimes g)}{\zeta(2s) \sqrt{L(1,\ad f) L(1,\ad g)}} \int_{0}^{\infty} W_{\frac{k}{2},\frac{\ell - 1}{2}}(y) y^{s + \frac{k}{2} - 1} e^{-\frac{y}{2}} \, \frac{dy}{y},
\end{align*}
recalling the definition \eqref{eqn:RSdefeq} of $L(s,f \otimes g)$ for $\Re(s) > 1$. From \cite[7.621.11]{GR15},
\[\int_{0}^{\infty} W_{\frac{k}{2},\frac{\ell - 1}{2}}(y) y^{s + \frac{k}{2} - 1} e^{-\frac{y}{2}} \, \frac{dy}{y} = \frac{\Gamma\left(s + \frac{k + \ell}{2} - 1\right) \Gamma\left(s + \frac{k - \ell}{2}\right)}{\Gamma(s)}.\]
Recalling the definition \eqref{eqn:Ckell} of $C_{k,\ell}$, we obtain the desired identity for all $s$ in the closed half-plane $\Re(s) \geq \frac{1}{2}$ with $s \neq 1$ by analytic continuation.
\end{proof}

\begin{lemma}
\label{lem:firstexpansionMaass}
Let $g$ be a holomorphic Hecke cusp form of weight $k$, and let $F_k$ be an $L^2$-normalised shifted Hecke--Maa\ss{} cusp form $F_k$ of weight $k$ arising from an unshifted Hecke--Maa\ss{} cusp form $f$ of weight $0$. Then for $\Re(s) \geq \frac{1}{2}$ with $s \neq 1$,
\begin{multline*}
\langle F_k, G \overline{E(\cdot,s)}\rangle = (-1)^{\frac{k}{2}} 2^{1 - 2s} \pi^{2 - s} \sqrt{\frac{\Gamma\left(\frac{1}{2} + it_f\right)}{\Gamma\left(\frac{1}{2} - it_f\right) \Gamma(k)}} \frac{\Gamma\left(s + \frac{k - 1}{2} + it_f\right) \Gamma\left(s + \frac{k - 1}{2} - it_f\right)}{\Gamma\left(\frac{k + 1}{2} + it_f\right) \Gamma(s)}	\\
\times \frac{L(s,f \otimes g)}{\zeta(2s) \sqrt{L(1,\ad f) L(1,\ad g)}}.
\end{multline*}
\end{lemma}

\begin{proof}
The $L^2$-normalised shifted Hecke--Maa\ss{} cusp form $F_k$ has the Fourier expansion
\begin{equation}
\label{eqn:FkMaassFourier}
F_k(z) = \sum_{\substack{n = -\infty \\ n \neq 0}}^{\infty} D_{k,t_f}^{\sgn(n)} \rho_f(n) W_{\sgn(n)\frac{k}{2},it_f}(4\pi|n|y) e(nx),
\end{equation}
where
\begin{gather}
\label{eqn:Dktpm}
D_{k,t_f}^{\pm} \coloneqq (-1)^{\frac{k}{2}} \frac{\Gamma\left(\frac{1}{2} + it_f\right)}{\Gamma\left(\frac{1 \pm k}{2} + it_f\right)},	\\
\label{eqn:rhof(n)}
\rho_f(n) = \rho_f(1) \sgn(n)^{\kappa_f} \frac{\lambda_f(|n|)}{\sqrt{|n|}}, \qquad |\rho_f(1)|^2 = \frac{\pi}{2 \Gamma\left(\frac{1}{2} + it_f\right) \Gamma\left(\frac{1}{2} - it_f\right) L(1,\ad f)}
\end{gather}
\cite[Section 4A]{BHMWW24}; here $\kappa_f \coloneqq (-1)^{\epsilon_f}$, which is $0$ if $f$ is an even Maa\ss{} form and $1$ if $f$ is an odd Maa\ss{} form. Again, the Hecke eigenvalues $\lambda_f(n)$ are real-valued, and we may normalise $f$ such that $\rho_f(1)$ is positive, so that the Fourier coefficients $\rho_f(n)$ are also real-valued. By unfolding $\langle F_k, G \overline{E(\cdot,s)}\rangle$ via \eqref{eqn:Eunfold}, inserting the Fourier expansions of $F_k$ and $g$ given by \eqref{eqn:FkMaassFourier} and \eqref{eqn:gFourier}, and using the identities \eqref{eqn:rhof(n)} and \eqref{eqn:rhog(n)} for the Fourier coefficients, we deduce that for $\Re(s) > 1$,
\begin{align*}
\langle F_k,G\overline{E(\cdot,s)}\rangle & = \int_{0}^{\infty} y^{s + \frac{k}{2}} \int_{0}^{1} F_k(x + iy) \overline{g(x + iy)} \, dx \, \frac{dy}{y^2}	\\
& = D_{k,t_f}^{+} \sum_{n = 1}^{\infty} \rho_f(n) \rho_g(n) (4\pi n)^{k/2} \int_{0}^{\infty} W_{\frac{k}{2},it_f}(4\pi ny) y^{s + \frac{k}{2} - 1} e^{-2\pi n y} \, \frac{dy}{y}	\\
& = D_{k,t_f}^{+} (4\pi)^{1 - s} \rho_f(1) \rho_g(1) \sum_{n = 1}^{\infty} \frac{\lambda_f(n) \lambda_g(n)}{n^s} \int_{0}^{\infty} W_{\frac{k}{2},it_f}(y) y^{s + \frac{k}{2} - 1} e^{-\frac{y}{2}} \, \frac{dy}{y}	\\
& = 2^{1 - 2s} \pi^{2 - s} \frac{D_{k,t_f}^{+}}{\sqrt{\Gamma\left(\frac{1}{2} + it_f\right) \Gamma\left(\frac{1}{2} - it_f\right) \Gamma(k)}} \frac{L(s,f \otimes g)}{\zeta(2s) \sqrt{L(1,\ad f) L(1,\ad g)}}	\\
& \hspace{7cm} \times \int_{0}^{\infty} W_{\frac{k}{2},it_f}(y) y^{s + \frac{k}{2} - 1} e^{-\frac{y}{2}} \, \frac{dy}{y}.
\end{align*}
From \cite[7.621.11]{GR15},
\[\int_{0}^{\infty} W_{\frac{k}{2},it_f}(y) y^{s + \frac{k}{2} - 1} e^{-\frac{y}{2}} \, \frac{dy}{y} = \frac{\Gamma\left(s + \frac{k - 1}{2} + it_f\right) \Gamma\left(s + \frac{k - 1}{2} - it_f\right)}{\Gamma(s)}.\]
Recalling the definition \eqref{eqn:Dktpm} of $D_{k,t_f}^{+}$, we obtain the desired identity.
\end{proof}

\begin{lemma}
\label{lem:firstexpansionEis}
Let $g$ be a holomorphic Hecke cusp form of weight $k$, and let $E_k(z,\frac{1}{2} + it)$ be the normalised shifted Eisenstein series of weight $k$ arising from the real analytic Eisenstein series $E(z,\frac{1}{2} + it)$. Then for $\Re(s) \geq \frac{1}{2}$ with $s \neq 1$,
\begin{multline*}
\left\langle E_k\left(\cdot,\frac{1}{2} + it\right),G\overline{E(\cdot,s)}\right\rangle = \frac{(-1)^{\frac{k}{2}} 2^{\frac{3}{2} - 2s} \pi^{2 - s + it}}{\sqrt{\Gamma(k)}} \frac{\Gamma\left(s + \frac{k - 1}{2} + it\right) \Gamma\left(s + \frac{k - 1}{2} - it\right)}{\Gamma\left(\frac{k + 1}{2} + it\right) \Gamma(s)}	\\
\times \frac{L(s + it,g) L(s - it,g)}{\zeta(2s) \zeta(1 + 2it) \sqrt{L(1,\ad g)}}.
\end{multline*}
\end{lemma}

\begin{proof}
The shifted Eisenstein series $E_k(z,\frac{1}{2} + it)$ has the Fourier expansion
\begin{multline}
\label{eqn:EFourier}
E_k\left(z,\frac{1}{2} + it\right) = y^{\frac{1}{2} + it} + (-1)^{\frac{k}{2}} \frac{\Gamma\left(\frac{1}{2} + it\right)^2}{\Gamma\left(\frac{k + 1}{2} + it\right) \Gamma\left(\frac{1 - k}{2} + it\right)} \frac{\xi(1 - 2it)}{\xi(1 + 2it)} y^{\frac{1}{2} - it}	\\
+ \sum_{\substack{n = -\infty \\ n \neq 0}}^{\infty} D_{k,t}^{\sgn(n)} \rho\left(n,\frac{1}{2} + it\right) W_{0,it}(4\pi|n|y) e(nx),
\end{multline}
where $D_{k,t}^{\pm}$ is as in \eqref{eqn:Dktpm}, while
\begin{equation}
\label{eqn:rho(n,s)}
\rho(n,s) = \frac{1}{\sqrt{|n|} \xi(2s)} \sum_{ab = |n|} a^{\frac{1}{2} - s} b^{s - \frac{1}{2}}, \qquad \xi(s) \coloneqq \pi^{-\frac{s}{2}} \Gamma\left(\frac{s}{2}\right) \zeta(s)
\end{equation}
\cite[Section 4A]{BHMWW24}. 
By unfolding $\langle E_k(\cdot,\frac{1}{2} + it), G \overline{E(\cdot,s)}\rangle$ via \eqref{eqn:Eunfold}, inserting the Fourier expansions of $E_k(\cdot,\frac{1}{2} + it)$ and $g$ given by \eqref{eqn:EFourier} and \eqref{eqn:gFourier}, and using the identities \eqref{eqn:rho(n,s)} and \eqref{eqn:rhog(n)} for the Fourier coefficients, we deduce that for $\Re(s) > 1$,
\begin{align*}
\left\langle E_k\left(\cdot,\frac{1}{2} + it\right), G\overline{E(\cdot,s)}\right\rangle \hspace{-2cm} & \hspace{2cm} = \int_{0}^{\infty} y^{s + \frac{k}{2}} \int_{0}^{1} E_k\left(x + iy,\frac{1}{2} + it\right) \overline{g(x + iy)} \, dx \, \frac{dy}{y^2}	\\
& = D_{k,t}^{+} \sum_{n = 1}^{\infty} \rho\left(n,\frac{1}{2} + it\right) \rho_g(n) (4\pi n)^{k/2} \int_{0}^{\infty} W_{\frac{k}{2},it}(4\pi ny) y^{s + \frac{k}{2} - 1} e^{-2\pi n y} \, \frac{dy}{y}	\\
& = D_{k,t}^{+} (4\pi)^{1 - s} \frac{\rho_g(1)}{\xi(1 + 2it)} \sum_{n = 1}^{\infty} \frac{\lambda_g(n) \sum_{ab = n} a^{-it} b^{it}}{n^s} \int_{0}^{\infty} W_{\frac{k}{2},it}(y) y^{s + \frac{k}{2} - 1} e^{-\frac{y}{2}} \, \frac{dy}{y}	\\
& = (-1)^{\frac{k}{2}} 2^{\frac{3}{2} - 2s} \pi^{2 - s + it} \frac{1}{\sqrt{\Gamma(k)}} \frac{\Gamma\left(s + \frac{k - 1}{2} + it\right) \Gamma\left(s + \frac{k - 1}{2} - it\right)}{\Gamma\left(\frac{k + 1}{2} + it\right) \Gamma(s)}	\\
& \hspace{5.5cm} \times \frac{L(s + it,g) L(s - it,g)}{\zeta(2s) \zeta(1 + 2it) \sqrt{L(1,\ad g)}}.
\end{align*}
Here we have once more invoked \cite[7.621.11]{GR15} to evaluate the integral over $\R_+ \ni y$, while we have additionally used the fact that for $\Re(s) > 1$ and $t \in \R$,
\[L(s + it,g) L(s - it,g) = \zeta(2s) \sum_{n = 1}^{\infty} \frac{\lambda_g(n) \sum_{ab = n} a^{-it} b^{it}}{n^s}.\qedhere\]
\end{proof}

With these inner product identities in hand, we proceed to the proof of \hyperref[prop:firstexpansion]{Proposition \ref*{prop:firstexpansion}}.

\begin{proof}[Proof of {\hyperref[prop:firstexpansion]{Proposition \ref*{prop:firstexpansion}}}]
The fact that the normalised inner product \eqref{eqn:firstnormalisedinnerproduct} extends holomorphically to all of $\C^2$ follows by writing this as
\[\frac{1}{\pi^3} \frac{\Gamma(k)^2}{\Gamma\left(k - \frac{1}{2}\right)^2} L(1,\ad g) \int_{\SL_2(\Z) \backslash \Hb} |G(z)|^2 E^{\ast}(z,s_1) E^{\ast}(z,s_2) \, d\mu(z)\]
and using the fact that $E^{\ast}(z,s)$ is entire.

Next, by \hyperref[lem:firstexpansionhol]{Lemma \ref*{lem:firstexpansionhol}}, the product of
\begin{equation}
\label{eqn:renormalise}
\pi^{3 - s_1 - s_2} \frac{\Gamma(k)^2}{\Gamma\left(k - \frac{1}{2}\right)^2} \Gamma(s_1) \Gamma(s_2) \zeta(2s_1) \zeta(2s_2) L(1,\ad g)
\end{equation}
and the first term on the right-hand side of \eqref{eqn:GEGE} is
\begin{multline}
\label{eqn:MMhols1s2}
\frac{4^{1 - s_1 - s_2}}{\pi} \frac{\Gamma(k)}{\Gamma\left(k - \frac{1}{2}\right)^2} \sum_{\substack{\ell = 2 \\ \ell \equiv 0 \hspace{-.25cm} \pmod{2}}}^{k} \sum_{\substack{f \in \BB_{\hol} \\ k_f = \ell}} \frac{L(s_1,f \otimes g) L(s_2,f \otimes g)}{L(1,\ad f)}	\\
\times \frac{\Gamma\left(s_1 + \frac{k + \ell}{2} - 1\right) \Gamma\left(s_1 + \frac{k - \ell}{2}\right) \Gamma\left(s_2 + \frac{k + \ell}{2} - 1\right) \Gamma\left(s_2 + \frac{k - \ell}{2}\right)}{\Gamma\left(\frac{k + \ell}{2}\right) \Gamma\left(\frac{k - \ell}{2} + 1\right)}.
\end{multline}
This is a holomorphic function of $(s_1,s_2) \in \C^2$, since the completed Rankin--Selberg $L$-function
\[\Lambda(s,f \otimes g) \coloneqq 4 (2\pi)^{-2s - k + 1} \Gamma\left(s + \frac{k + \ell}{2} - 1\right) \Gamma\left(s + \frac{k - \ell}{2}\right) L(s,f \otimes g)\]
is entire, and the value of \eqref{eqn:MMhols1s2} at $(s_1,s_2) = (\frac{1}{2},\frac{1}{2})$ is precisely $\MM_{\hol}$.

By \hyperref[lem:firstexpansionMaass]{Lemma \ref*{lem:firstexpansionMaass}}, the product of \eqref{eqn:renormalise} and the second term is
\begin{multline}
\label{eqn:MMMaasss1s2}
\frac{4^{1 - s_1 - s_2}}{\pi} \frac{\Gamma(k)}{\Gamma\left(k - \frac{1}{2}\right)^2} \sum_{f \in \BB_0} \frac{L(s_1,f \otimes g) L(s_2,f \otimes g)}{L(1,\ad f)}	\\
\times \frac{\Gamma\left(s_1 + \frac{k - 1}{2} + it_f\right) \Gamma\left(s_1 + \frac{k - 1}{2} - it_f\right) \Gamma\left(s_2 + \frac{k - 1}{2} + it_f\right) \Gamma\left(s_2 + \frac{k - 1}{2} - it_f\right)}{\Gamma\left(\frac{k + 1}{2} + it_f\right) \Gamma\left(\frac{k + 1}{2} - it_f\right)}.
\end{multline}
A straightforward application of Stirling's approximation (see \eqref{eqn:Stirling} below) shows that the second line of \eqref{eqn:MMMaasss1s2} is $\ll_{\Re(s_1),\Re(s_2),k,A} t_f^{-A}$ for any $A \geq 0$. Moreover, as the completed Rankin--Selberg $L$-function
\[\Lambda(s,f \otimes g) \coloneqq 4 (2\pi)^{-2s - k + 1} \Gamma\left(s + \frac{k - 1}{2} + it_f\right) \Gamma\left(s + \frac{k - 1}{2} - it_f\right) L(s,f \otimes g)\]
is entire, the convexity bound for $L(s,f \otimes g)$ together with the Weyl law implies that \eqref{eqn:MMMaasss1s2} is a holomorphic function of $(s_1,s_2) \in \C^2$, and the value of \eqref{eqn:MMMaasss1s2} at $(s_1,s_2) = (\frac{1}{2},\frac{1}{2})$ is precisely $\MM_{\Maass}$.

Finally, by \hyperref[lem:firstexpansionEis]{Lemma \ref*{lem:firstexpansionEis}}, the product of \eqref{eqn:renormalise} and the third term is
\begin{multline*}
\frac{1}{2\pi} \frac{4^{1 - s_1 - s_2}}{\pi} \frac{\Gamma(k)}{\Gamma\left(k - \frac{1}{2}\right)^2} \int_{-\infty}^{\infty} \frac{L(s_1 + it,g) L(s_1 - it,g) L(s_2 + it,g) L(s_2 - it,g)}{\zeta(1 + 2it) \zeta(1 - 2it)}	\\
\times \frac{\Gamma\left(s_1 + \frac{k - 1}{2} + it\right) \Gamma\left(s_1 + \frac{k - 1}{2} - it\right) \Gamma\left(s_2 + \frac{k - 1}{2} + it\right) \Gamma\left(s_2 + \frac{k - 1}{2} - it\right)}{\Gamma\left(\frac{k + 1}{2} + it\right) \Gamma\left(\frac{k + 1}{2} - it\right)} \, dt.
\end{multline*}
Once more, as the completed $L$-function
\[\Lambda(s,g) \coloneqq 2 (2\pi)^{-s - \frac{k - 1}{2}} \Gamma\left(s + \frac{k - 1}{2}\right) L(s,g)\]
is entire, this is a holomorphic function of $(s_1,s_2) \in \C^2$, and its value at $(s_1,s_2) = (\frac{1}{2},\frac{1}{2})$ is precisely $\MM_{\Eis}$.
\end{proof}

\subsection{The Second Spectral Expansion}

We next consider the spectral expansion of the inner product $\langle E^{\ast}(\cdot,s_1) E^{\ast}(\cdot,s_2), |G|^2\rangle$.

\begin{proposition}
\label{prop:secondexpansion}
Let $g$ be a holomorphic Hecke cusp form of positive even weight $k$ and let $G(z) \coloneqq \Im(z)^{k/2} g(z)$. The normalised inner product
\begin{equation}
\label{eqn:secondnormalisedinnerproduct}
\pi^{-3 - s_1 - s_2} \frac{\Gamma(k)^2}{\Gamma\left(k - \frac{1}{2}\right)^2} \Gamma(s_1) \Gamma(s_2) \zeta(2s_1) \zeta(2s_2) L(1,\ad g) \left\langle E(\cdot,s_1) E(\cdot,s_2),|G|^2\right\rangle
\end{equation}
is absolutely convergent for $\Re(s_1),\Re(s_2) > 1$ and extends holomorphically to all of $\C^2$. At $(s_1,s_2) = (\frac{1}{2},\frac{1}{2})$, it is equal to
\[\widetilde{\MM}_0 + \widetilde{\MM}_{\Maass} + \widetilde{\MM}_{\Eis},\]
with $\widetilde{\MM}_0$, $\widetilde{\MM}_{\Maass}$, and $\widetilde{\MM}_{\Eis}$ as in \eqref{eqn:tildeMM0defeq}, \eqref{eqn:tildeMMMaassdefeq}, and \eqref{eqn:tildeMMEisdefeq} respectively.
\end{proposition}

Since \hyperref[prop:firstexpansion]{Propositions \ref*{prop:firstexpansion}} and \ref{prop:secondexpansion} give two different identities for the same normalised inner product at the value $(s_1,s_2) = (\frac{1}{2},\frac{1}{2})$, they combine to complete the proof of \hyperref[thm:exactidentity]{Theorem \ref*{thm:exactidentity}}.

\begin{proof}[Proof of {\hyperref[thm:exactidentity]{Theorem \ref*{thm:exactidentity}}}]
This follows immediately from combining \hyperref[prop:firstexpansion]{Propositions \ref*{prop:firstexpansion}} and \ref{prop:secondexpansion}.
\end{proof}

Once more, the first step of the proof of \hyperref[prop:secondexpansion]{Proposition \ref*{prop:secondexpansion}} is to apply Parseval's identity. There is a subtlety in this regard: Eisenstein series are not square-integrable. Instead, we use a regularisation technique due to Zagier \cite{Zag82}. We let
\begin{multline*}
\EE(z) \coloneqq E(z,s_1 + s_2) + \frac{\xi(2 - 2s_1)}{\xi(2s_1)} E(z,1 - s_1 + s_2) + \frac{\xi(2 - 2s_2)}{\xi(2s_2)} E(z,1 + s_1 - s_2)	\\
+ \frac{\xi(2 - 2s_1)}{\xi(2s_1)} \frac{\xi(2 - 2s_2)}{\xi(2s_2)} E(z,2 - s_1 - s_2).
\end{multline*}
This linear combination of Eisenstein series is chosen such that for $\frac{1}{2} < \Re(s_1),\Re(s_2) < \frac{3}{4}$ with $s_1 \neq s_2$, the constant terms of these Eisenstein series that grow at least as fast as $\Im(z)^{1/2}$ as $\Im(z)$ tends to infinity exactly cancels out the constant terms of the product of Eisenstein series $E(z,s_1) E(z,s_2)$. Consequently, there exists some $\delta > 0$ such that
\[E(z,s_1) E(z,s_2) - \EE(z) = O\left(\Im(z)^{\frac{1}{2} - \delta}\right)\]
as $\Im(z)$ tends to infinity, and so $E(\cdot,s_1) E(\cdot,s_2) - \EE$ \emph{is} square-integrable.

\begin{lemma}
\label{lem:EEGG}
For $\frac{1}{2} < \Re(s_1),\Re(s_2) < \frac{3}{4}$ with $s_1 \neq s_2$, we have that
\begin{multline}
\label{eqn:EEGG}
\left\langle E(\cdot,s_1) E(\cdot,s_2),|G|^2\right\rangle = \frac{3}{\pi} \left\langle E(\cdot,s_1) E(\cdot,s_2),1\right\rangle_{\reg} \left\langle 1,|G|^2\right\rangle + \left\langle \EE,|G|^2\right\rangle	\\
+ \sum_{f \in \BB_0} \left\langle E(\cdot,s_1) E(\cdot,s_2),f\right\rangle \left\langle f,|G|^2\right\rangle	\\
+ \frac{1}{4\pi} \int_{-\infty}^{\infty} \left\langle E(\cdot,s_1) E(\cdot,s_2),E\left(\cdot,\frac{1}{2} + it\right)\right\rangle_{\reg} \left\langle E\left(\cdot,\frac{1}{2} + it\right),|G|^2\right\rangle \, dt.
\end{multline}
\end{lemma}

Here $\langle \cdot,\cdot\rangle_{\reg}$ denotes the \emph{regularised} inner product in the sense of Zagier \cite{Zag82}.

\begin{proof}
This follows from the regularised form of Parseval's identity for automorphic forms of weight $0$ \cite[Lemma 4.1]{You16}. 
\end{proof}

We first treat the main terms, namely the first two terms on the right-hand side of \eqref{eqn:EEGG}. Since $G$ is $L^2$-normalised, we have that $\langle 1,|G|^2\rangle = 1$, while the following identity then implies that the first main term vanishes.

\begin{lemma}[Zagier {\cite[p.428]{Zag82}}]
\label{lem:secondexpansionfirstmainterm}
For $\frac{1}{2} < \Re(s_1),\Re(s_2) < 1$ with $s_1 \neq s_2$, we have that
\[\left\langle E(\cdot,s_1) E(\cdot,s_2),1\right\rangle_{\reg} = 0.\]
\end{lemma}

For the second main term, we have the following.

\begin{lemma}
\label{lem:secondexpansionsecondmainterm}
For $\frac{1}{2} < \Re(s_1),\Re(s_2) < \frac{3}{4}$ with $s_1 \neq s_2$, we have that
\begin{multline*}
\left\langle \EE,|G|^2\right\rangle = 2^{1 - 2s_1 - 2s_2} \pi^{2 - s_1 - s_2} \frac{\Gamma(s_1 + s_2 + k - 1)}{\Gamma(k)} \frac{\zeta(s_1 + s_2) L(s_1 + s_2,\ad g)}{\zeta(2s_1 + 2s_2) L(1,\ad g)}	\\
+ 2^{2s_1 - 2s_2 - 1} \pi^{3s_1 - s_2} \frac{\Gamma(s_2 - s_1 + k) \Gamma(1 - s_1)}{\Gamma(k) \Gamma(s_1)} \frac{\zeta(2 - 2s_1) \zeta(1 - s_1 + s_2) L(1 - s_1 + s_2,\ad g)}{\zeta(2s_1) \zeta(2 - 2s_1 + 2s_2) L(1,\ad g)}	\\
+ 2^{2s_2 - 2s_1 - 1} \pi^{3s_2 - s_1} \frac{\Gamma(s_1 - s_2 + k) \Gamma(1 - s_2)}{\Gamma(k) \Gamma(s_2)} \frac{\zeta(2 - 2s_2) \zeta(1 + s_1 - s_2) L(1 + s_1 - s_2,\ad g)}{\zeta(2s_2) \zeta(2 + 2s_1 - 2s_2) L(1,\ad g)}	\\
+ 2^{2s_1 + 2s_2 - 3} \pi^{3s_1 + 3s_2 - 2} \frac{\Gamma(1 - s_1 - s_2 + k) \Gamma(1 - s_1) \Gamma(1 - s_2)}{\Gamma(k) \Gamma(s_1) \Gamma(s_2)}	\\
\times \frac{\zeta(2 - 2s_1) \zeta(2 - 2s_2) \zeta(2 - s_1 - s_2) L(2 - s_1 - s_2,\ad g)}{\zeta(2s_1) \zeta(2s_2) \zeta(4 - 2s_1 - 2s_2) L(1,\ad g)}.
\end{multline*}
\end{lemma}

\begin{proof}
This follows from the definition of $\EE$ together with the fact that
\[\left\langle E(\cdot,s),|G|^2\right\rangle = 2^{1 - 2s} \pi^{2 - s} \frac{\Gamma(s + k - 1)}{\Gamma(k)} \frac{\zeta(s) L(s,\ad g)}{\zeta(2s) L(1,\ad g)}\]
from \hyperref[lem:firstexpansionhol]{Lemma \ref*{lem:firstexpansionhol}} with $F_k = G$, so that $\ell = k$, together with the definition \eqref{eqn:rho(n,s)} of $\xi(s)$.
\end{proof}

We next deal with the third term on the right-hand side of \eqref{eqn:EEGG}. The first inner product appearing in the third term may be dealt with via the Rankin--Selberg method.

\begin{lemma}
\label{lem:secondexpansionEEf}
For $\Re(s_1),\Re(s_2) \geq \frac{1}{2}$ with $s_1,s_2 \neq 1$, we have that
\begin{multline*}
\left\langle E(\cdot,s_1) E(\cdot,s_2),f\right\rangle = \frac{1 + \epsilon_f}{2} 2^{-\frac{1}{2}} \pi^{s_2 - s_1 + \frac{1}{2}} \frac{L\left(s_1 + s_2 - \frac{1}{2},f\right) L\left(s_1 - s_2 + \frac{1}{2},f\right)}{\zeta(2s_1) \zeta(2s_2) \sqrt{L(1,\ad f)}}	\\
\times \frac{\Gamma\left(\frac{s_1 + s_2 - \frac{1}{2} + it_f}{2}\right) \Gamma\left(\frac{s_1 + s_2 - \frac{1}{2} - it_f}{2}\right) \Gamma\left(\frac{s_1 - s_2 + \frac{1}{2} + it_f}{2}\right) \Gamma\left(\frac{s_1 - s_2 + \frac{1}{2} - it_f}{2}\right)}{\Gamma(s_1) \Gamma(s_2) \sqrt{\Gamma\left(\frac{1}{2} + it_f\right) \Gamma\left(\frac{1}{2} - it_f\right)}}.
\end{multline*}
\end{lemma}

\begin{proof}
The Eisenstein series $E(z,s)$ has the Fourier expansion
\begin{equation}
\label{eqn:EFourier2}
E(z,s) = y^s + \frac{\xi(2 - 2s)}{\xi(2s)} y^{1 - s} + \sum_{\substack{n = -\infty \\ n \neq 0}}^{\infty} \rho(n,s) W_{0,s - \frac{1}{2}}(4\pi|n|y) e(nx).
\end{equation}
By unfolding via \eqref{eqn:Eunfold}, inserting the Fourier expansions of $f$ and $E(\cdot,s_2)$ given by \eqref{eqn:FkMaassFourier} and \eqref{eqn:EFourier2}, and using the identities \eqref{eqn:rhof(n)} and \eqref{eqn:rho(n,s)} for the Fourier coefficients, we find that for $\Re(s_1),\Re(s_2) > 1$ and $\Re(s_1) - \Re(s_2) > \frac{1}{2}$,
\begin{align*}
\left\langle E(\cdot,s_1) E(\cdot,s_2),f\right\rangle & = \int_{0}^{\infty} y^{s_1} \int_{0}^{1} \overline{f(x + iy)} E(x + iy,s_2) \, dx \, \frac{dy}{y^2}	\\
& = \sum_{\substack{n = -\infty \\ n \neq 0}}^{\infty} \rho_f(n) \rho(n,s_2) \int_{0}^{\infty} W_{0,it_f}(4\pi|n|y) W_{0,s_2 - \frac{1}{2}}(4\pi|n|y) y^{s_1 - 1} \, \frac{dy}{y}	\\
& = (1 + \epsilon_f) (4\pi)^{1 - s_1} \frac{\rho_f(1)}{\xi(2s_2)} \sum_{n = 1}^{\infty} \frac{\lambda_f(n) \sum_{ab = n} a^{\frac{1}{2} - s_2} b^{s_2 - \frac{1}{2}}}{n^{s_1}}	\\
& \hspace{6cm} \times \int_{0}^{\infty} W_{0,it_f}(y) W_{0,s_2 - \frac{1}{2}}(y) y^{s_1 - 1} \, \frac{dy}{y}	\\
& = \frac{1 + \epsilon_f}{2} \frac{2^{\frac{5}{2} - 2s_1} \pi^{s_2 - s_1 + \frac{3}{2}}}{\Gamma(s_2) \sqrt{\Gamma\left(\frac{1}{2} + it_f\right) \Gamma\left(\frac{1}{2} - it_f\right)}} \frac{L\left(s_1 + s_2 - \frac{1}{2},f\right) L\left(s_1 - s_2 + \frac{1}{2},f\right)}{\zeta(2s_1) \zeta(2s_2) \sqrt{L(1,\ad f)}}	\\
& \hspace{6cm} \times \int_{0}^{\infty} W_{0,it_f}(y) W_{0,s_2 - \frac{1}{2}}(y) y^{s_1 - 1} \, \frac{dy}{y}.
\end{align*}
By \cite[6.576.4 and 9.235.2]{GR15}, we have that
\begin{multline*}
\int_{0}^{\infty} W_{0,it_f}(y) W_{0,s_2 - \frac{1}{2}}(y) y^{s_1 - 1} \, \frac{dy}{y}	\\
= \frac{2^{2s_1 - 3}}{\pi} \frac{\Gamma\left(\frac{s_1 + s_2 + it_f - \frac{1}{2}}{2}\right) \Gamma\left(\frac{s_1 + s_2 - it_f - \frac{1}{2}}{2}\right) \Gamma\left(\frac{s_1 - s_2 + it_f + \frac{1}{2}}{2}\right) \Gamma\left(\frac{s_1 - s_2 - it_f + \frac{1}{2}}{2}\right)}{\Gamma(s_1)}.
\end{multline*}
By analytic continuation, this yields the desired identity in the region for which $\Re(s_1),\Re(s_2) \geq \frac{1}{2}$ with $s_1,s_2 \neq 1$ and $\Re(s_1) - \Re(s_2) > - \frac{1}{2}$.
\end{proof}

For the second inner product appearing in the third term on the right-hand side of \eqref{eqn:EEGG}, the Rankin--Selberg method is no longer applicable, since there are no Eisenstein series present. Instead, we use the Watson--Ichino triple product formula.

\begin{lemma}
\label{lem:secondexpansionfGG}
There exists a complex constant $\theta_{f,g}$ of absolute value $1$ dependent only on $f$ and $g$ such that
\begin{equation}
\label{eqn:fGG}
\left\langle f,|G|^2\right\rangle = \theta_{f,g} \frac{1 + \epsilon_f}{2} 2^{-\frac{1}{2} - 2it_f} \pi^{\frac{3}{2} - 2it_f} \sqrt{\frac{\Gamma\left(\frac{1}{2} + it_f\right)}{\Gamma\left(\frac{1}{2} - it_f\right)}} \frac{\Gamma\left(k - \frac{1}{2} + it_f\right)}{\Gamma(k)} \frac{\sqrt{L\left(\frac{1}{2},\ad g \otimes f\right) L\left(\frac{1}{2},f\right)}}{\sqrt{L(1,\ad f)} L(1,\ad g)}.
\end{equation}
\end{lemma}

\begin{proof}
This result is a consequence of the Watson--Ichino triple product formula \cite{Ich08,Wats02}. A similar identity is proven in \cite[Lemma 9.1]{BHMWW24}, namely an exact identity for $|\langle |f|^2,G\rangle|^2$. The same method combined with the uniqueness of trilinear functionals via \cite[Theorem 9.3]{Pra90} and \cite[Theorem 1.2]{Lok01} (see also \cite[Theorem 11.1]{JSZ11}) implies that there exists some complex constant $\theta_{f,g}$ of absolute value $1$ dependent only on $f$ and $g$ such that
\[\left\langle f,|G|^2\right\rangle = \theta_{f,g} \frac{1 + \epsilon_f}{2} \sqrt{L\left(\frac{1}{2},f\right) L\left(\frac{1}{2},\ad g \otimes f\right)} \int_{\R^{\times}} W_1\begin{pmatrix} y & 0 \\ 0 & 1 \end{pmatrix} W_2\begin{pmatrix} y & 0 \\ 0 & 1 \end{pmatrix} f_3\begin{pmatrix} y & 0 \\ 0 & 1 \end{pmatrix} \, \frac{dy}{|y|^2},\]
where for $y \in \R^{\times}$,
\begin{align*}
W_1\begin{pmatrix} y & 0 \\ 0 & 1 \end{pmatrix} = W_2\begin{pmatrix} y & 0 \\ 0 & 1 \end{pmatrix} & = \begin{dcases*}
\rho_g(1) (4\pi y)^{\frac{k}{2}} e^{-2\pi y} & if $y > 0$,	\\
0 & if $y < 0$,
\end{dcases*}	\\
f_3\begin{pmatrix} y & 0 \\ 0 & 1 \end{pmatrix} & = \pi^{-\frac{1}{2} - it_f} \Gamma\left(\frac{1}{2} + it_f\right) \sgn(y)^{\kappa_f} \rho_f(1) |y|^{\frac{1}{2} + it_f}.
\end{align*}
It follows via the identities \eqref{eqn:rhog(n)} and \eqref{eqn:rhof(n)} that
\begin{align*}
\int_{\R^{\times}} W_1\begin{pmatrix} y & 0 \\ 0 & 1 \end{pmatrix} W_2\begin{pmatrix} y & 0 \\ 0 & 1 \end{pmatrix} f_3\begin{pmatrix} y & 0 \\ 0 & 1 \end{pmatrix} \, \frac{dy}{|y|^2} \hspace{-4cm} &	\\
& = \pi^{-\frac{1}{2} - it_f} \Gamma\left(\frac{1}{2} + it_f\right) (4\pi)^k \rho_f(1) \rho_g(1)^2 \int_{0}^{\infty} y^{k - \frac{1}{2} + it_f} e^{-4\pi y} \, \frac{dy}{y}	\\
& = 2^{-\frac{1}{2} - 2it_f} \pi^{\frac{3}{2} - 2it_f} \sqrt{\frac{\Gamma\left(\frac{1}{2} + it_f\right)}{\Gamma\left(\frac{1}{2} - it_f\right)}} \frac{\Gamma\left(k - \frac{1}{2} + it_f\right)}{\Gamma(k)} \frac{1}{\sqrt{L(1,\ad f)} L(1,\ad g)}.
\qedhere
\end{align*}
\end{proof}

For the fourth term on the right-hand side of \eqref{eqn:EEGG}, we invoke an identity of Zagier for the regularised integral of three Eisenstein series, which is a regularised form of the Rankin--Selberg method.

\begin{lemma}[Zagier {\cite[p.430]{Zag82}}]
\label{lem:secondexpansionEEE}
For $\frac{1}{2} < \Re(s_1),\Re(s_2) < \frac{3}{4}$ and $t \in \R$, we have that
\begin{multline*}
\left\langle E(\cdot,s_1) E(\cdot,s_2),E\left(\cdot,\frac{1}{2} + it\right)\right\rangle_{\reg}	\\
= \pi^{s_2 - s_1 + \frac{1}{2}} \frac{\Gamma\left(\frac{s_1 + s_2 - \frac{1}{2} + it}{2}\right) \Gamma\left(\frac{s_1 + s_2 - \frac{1}{2} - it}{2}\right) \Gamma\left(\frac{s_1 - s_2 + \frac{1}{2} + it}{2}\right) \Gamma\left(\frac{s_1 - s_2 + \frac{1}{2} - it}{2}\right)}{\Gamma(s_1) \Gamma(s_2) \Gamma\left(\frac{1}{2} - it\right)}	\\
\times \frac{\zeta\left(s_1 + s_2 - \frac{1}{2} + it\right) \zeta\left(s_1 + s_2 - \frac{1}{2} - it\right) \zeta\left(s_1 - s_2 + \frac{1}{2} + it\right) \zeta\left(s_1 - s_2 + \frac{1}{2} - it\right)}{\zeta(1 - 2it) \zeta(2s_1) \zeta(2s_2)}.
\end{multline*}
\end{lemma}

The remaining inner product of interest can be determined once more by the Rankin--Selberg method.

\begin{lemma}
\label{lem:secondexpansionEGG}
We have that
\[\left\langle E\left(\cdot,\frac{1}{2} + it\right),|G|^2\right\rangle = 2^{-2it} \pi^{\frac{3}{2} - 2it} \frac{\Gamma\left(k - \frac{1}{2} + it\right)}{\Gamma(k)} \frac{L\left(\frac{1}{2} + it,\ad g\right) \zeta\left(\frac{1}{2} + it\right)}{\zeta(1 + 2it) L(1,\ad g)}.\]
\end{lemma}

\begin{proof}
This follows from \hyperref[lem:firstexpansionhol]{Lemma \ref*{lem:firstexpansionhol}} with $s = \frac{1}{2} + it$ and $F_k = G$, so that $\ell = k$.
\end{proof}

With these inner product identities in hand, we proceed to the proof of \hyperref[prop:secondexpansion]{Proposition \ref*{prop:secondexpansion}}.

\begin{proof}[Proof of {\hyperref[prop:secondexpansion]{Proposition \ref*{prop:secondexpansion}}}]
Again, the fact that the normalised inner product \eqref{eqn:secondnormalisedinnerproduct} extends holomorphically to all of $\C^2$ follows by writing this as
\[\frac{1}{\pi^3} \frac{\Gamma(k)^2}{\Gamma\left(k - \frac{1}{2}\right)^2} L(1,\ad g) \int_{\SL_2(\Z) \backslash \Hb} |G(z)|^2 E^{\ast}(z,s_1) E^{\ast}(z,s_2) \, d\mu(z)\]
and using the fact that $E^{\ast}(z,s)$ is entire.

Next, by \hyperref[lem:secondexpansionfirstmainterm]{Lemma \ref*{lem:secondexpansionfirstmainterm}}, the product of \eqref{eqn:renormalise} and the first term on the right-hand side of \eqref{eqn:EEGG} is $0$. By \hyperref[lem:secondexpansionEEf]{Lemmata \ref*{lem:secondexpansionEEf}} and \ref{lem:secondexpansionfGG}, the product of \eqref{eqn:renormalise} and the third term is
\begin{multline}
\label{eqn:tildeMMMaasss1s2}
\frac{4^{1 - s_1 - s_2}}{\pi} \frac{\Gamma(k)}{\Gamma\left(k - \frac{1}{2}\right)^2}	\\
\times \sum_{\substack{f \in \BB_0 \\ \epsilon_f = 1}} \theta_{f,g} \frac{\sqrt{L\left(\frac{1}{2},\ad g \otimes f\right) L\left(\frac{1}{2},f\right)} L\left(s_1 + s_2 - \frac{1}{2},f\right) L\left(s_1 - s_2 + \frac{1}{2},f\right)}{L(1,\ad f)} 2^{2s_1 + 2s_2 - 3 - 2it_f} \pi^{-2 - 2it_f}	\\
\times \frac{\Gamma\left(k - \frac{1}{2} + it_f\right) \Gamma\left(\frac{s_1 + s_2 - \frac{1}{2} + it_f}{2}\right) \Gamma\left(\frac{s_1 + s_2 - \frac{1}{2} - it_f}{2}\right) \Gamma\left(\frac{s_1 - s_2 + \frac{1}{2} + it_f}{2}\right) \Gamma\left(\frac{s_1 - s_2 + \frac{1}{2} - it_f}{2}\right)}{\Gamma\left(\frac{1}{2} - it_f\right)}.
\end{multline}
A straightforward application of Stirling's approximation (see \eqref{eqn:Stirling} below) shows that the third line of \eqref{eqn:tildeMMMaasss1s2} is $\ll_{\Re(s_1),\Re(s_2),k,A} t_f^{-A}$ for any $A \geq 0$. Moreover, as the completed $L$-function
\[\Lambda(s,f) \coloneqq \pi^{-s} \Gamma\left(\frac{s}{2} + \frac{it_f}{4}\right) \Gamma\left(\frac{s}{2} - \frac{it_f}{4}\right) L(s,f)\]
is entire, the convexity bound for $L(\frac{1}{2},\ad g \otimes f)$ and $L(s,f)$ together with the Weyl law implies that \eqref{eqn:tildeMMMaasss1s2} is a holomorphic function of $(s_1,s_2) \in \C^2$, and the value of \eqref{eqn:tildeMMMaasss1s2} at $(s_1,s_2) = (\frac{1}{2},\frac{1}{2})$ is precisely $\widetilde{\MM}_{\Maass}$.

Similarly, by \hyperref[lem:secondexpansionEEE]{Lemmata \ref*{lem:secondexpansionEEE}} and \ref{lem:secondexpansionEGG}, the product of \eqref{eqn:renormalise} and the fourth term is
\begin{multline}
\label{eqn:tildeMMEiss1s2}
\frac{4^{1 - s_1 - s_2}}{\pi} \frac{\Gamma(k)}{\Gamma\left(k - \frac{1}{2}\right)^2}	\\
\times \frac{1}{2\pi} \int_{-\infty}^{\infty} \frac{L\left(\frac{1}{2} + it,\ad g\right) \zeta\left(\frac{1}{2} + it\right) \prod_{\pm} \zeta\left(s_1 + s_2 - \frac{1}{2} \pm it\right) \zeta\left(s_1 - s_2 + \frac{1}{2} \pm it\right)}{\zeta(1 + 2it) \zeta(1 - 2it)} 2^{2s_1 + 2s_2 - 3 - 2it} \pi^{-2 - 2it}	\\
\times \frac{\Gamma\left(k - \frac{1}{2} + it\right) \Gamma\left(\frac{s_1 + s_2 - \frac{1}{2} + it}{2}\right) \Gamma\left(\frac{s_1 + s_2 - \frac{1}{2} - it}{2}\right) \Gamma\left(\frac{s_1 - s_2 + \frac{1}{2} + it}{2}\right) \Gamma\left(\frac{s_1 - s_2 + \frac{1}{2} - it}{2}\right)}{\Gamma\left(\frac{1}{2} - it\right)} \, dt.
\end{multline}
Since the completed Riemann zeta function $\xi(s) \coloneqq \pi^{-s/2} \Gamma(\frac{s}{2}) \zeta(s)$ is holomorphic for $0 < \Re(s) < 1$, this is a holomorphic function of $(s_1,s_2)$ in the region for which $\frac{1}{2} < \Re(s_1) + \Re(s_2) < \frac{3}{2}$ and $-\frac{1}{2} < \Re(s_1) - \Re(s_2) < \frac{1}{2}$, and the value of \eqref{eqn:tildeMMEiss1s2} at $(s_1,s_2) = (\frac{1}{2},\frac{1}{2})$ is precisely $\widetilde{\MM}_{\Eis}$.

It remains to treat the product of \eqref{eqn:renormalise} and the second term. We introduce the variables $w_1 \coloneqq s_1 - \frac{1}{2}$ and $w_2 \coloneqq s_2 - \frac{1}{2}$. By \hyperref[lem:secondexpansionsecondmainterm]{Lemma \ref*{lem:secondexpansionsecondmainterm}}, the product of \eqref{eqn:renormalise} and the second term is
\begin{equation}
\label{eqn:tildeMM0w1w2defeq}
\widetilde{\MM}_0(w_1,w_2) \coloneqq \frac{\Gamma(k)^2}{\Gamma\left(k - \frac{1}{2}\right)^2} \sum_{\pm_1, \pm_2} \Psi(\pm_1 w_1, \pm_2 w_2),
\end{equation}
where
\begin{multline}
\label{eqn:Psiw1w2defeq}
\Psi(w_1,w_2) \coloneqq 2 (2\pi)^{-2w_1 - 2w_2} \frac{\Gamma(w_1 + w_2 + k) \Gamma\left(\frac{1}{2} + w_1\right) \Gamma\left(\frac{1}{2} + w_2\right)}{\Gamma(k)} \\ 
\times \frac{\zeta(1 + 2w_1) \zeta(1 + 2w_2) \zeta(1 + w_1 + w_2) L(1 + w_1 + w_2,\ad g)}{\zeta(2 + 2w_1 + 2w_2)}.
\end{multline}
Each function $\Psi(\pm_1 w_1, \pm_2 w_2)$ has polar lines at $w_1 = 0$, $w_2 = 0$, and $w_1 = \pm_1 \mp_2 w_2$. Nevertheless, the \emph{sum} of these functions has nicer behaviour. Indeed, from \hyperref[lem:EEGG]{Lemma \ref*{lem:EEGG}} and the identities \eqref{eqn:tildeMMMaasss1s2} and \eqref{eqn:tildeMMEiss1s2}, $\widetilde{\MM}_0(w_1,w_2)$ is equal to a sum of functions that are holomorphic in the variables $(w_1,w_2)$ in the region for which $-\frac{1}{2} < \Re(w_1) + \Re(w_2) < \frac{1}{2}$ and $-\frac{1}{2} < \Re(w_1) - \Re(w_2) < \frac{1}{2}$, and consequently $\widetilde{\MM}_0(w_1,w_2)$ must also be holomorphic in such a neighbourhood. In particular, $\lim_{(w_1,w_2) \to (0,0)} \widetilde{\MM}_0(w_1,w_2)$ exists, which gives the desired identity \eqref{eqn:tildeMM0defeq} for $\widetilde{\MM}_0$.
\end{proof}

\section{Bounds for Moments of \texorpdfstring{$L$}{L}-Functions}

\subsection{The Size of the Weight Functions}

We now prove bounds for the weight functions $h^{\hol}(\ell,k)$, $h(t,k)$, and $\widetilde{h}(t,k)$ given by \eqref{eqn:hholellkdefeq}, \eqref{eqn:htkdefeq}, and \eqref{eqn:tildehtkdefeq}. In doing so, we repeatedly make use of Stirling's approximation, which states that for $s = \sigma + i\tau$ with $\sigma \geq \delta > 0$ and $\tau \in \R$,
\begin{equation}
\label{eqn:Stirling}
\log|\Gamma(\sigma + i\tau)| = \frac{1}{2} \left(\sigma - \frac{1}{2}\right) \log (\sigma^2 + \tau^2) - |\tau| \arctan \frac{|\tau|}{\sigma} - \sigma + O_{\delta}(1).
\end{equation}

\begin{lemma}
\label{lem:hholellkbound}
For a positive even integer $\ell < k$, we have that
\[h^{\hol}(\ell,k) \ll \frac{1}{k - \ell} \left(\frac{k - \ell}{2(k - 1)}\right)^{\frac{k - \ell}{2}}.\]
\end{lemma}

\begin{proof}
For a positive even integer $\ell \leq k$, we have that
\[h^{\hol}(\ell,k) = \frac{\Gamma(k)^2}{\Gamma\left(k - \frac{1}{2}\right)^2} \frac{\Gamma\left(\frac{k + \ell - 1}{2}\right)^2}{\Gamma\left(\frac{k + \ell}{2}\right)^2} \frac{\Gamma\left(\frac{k - \ell + 1}{2}\right)^2}{\Gamma\left(\frac{k - \ell}{2} + 1\right)^2} \binom{k - 1}{\frac{k - \ell}{2}}^{-1}.\]
We then observe by Stirling's approximation \eqref{eqn:Stirling} that
\[\frac{\Gamma(k)^2}{\Gamma\left(k - \frac{1}{2}\right)^2} \asymp k, \qquad \frac{\Gamma\left(\frac{k + \ell - 1}{2}\right)^2}{\Gamma\left(\frac{k + \ell}{2}\right)^2} \asymp \frac{1}{k}, \qquad \frac{\Gamma\left(\frac{k - \ell + 1}{2}\right)^2}{\Gamma\left(\frac{k - \ell}{2} + 1\right)^2} \asymp \frac{1}{k - \ell + 1},\]
while for $\ell < k$, we have the inequality
\[\binom{k - 1}{\frac{k - \ell}{2}}^{-1} \leq \left(\frac{k - \ell}{2(k - 1)}\right)^{\frac{k - \ell}{2}}.\qedhere\]
\end{proof}

\begin{lemma}
\label{lem:htkbound}
For $t \in \R$, we have that
\[h(t,k) \ll \begin{dcases*}
\frac{1}{2^k \sqrt{k}} & if $|t| \leq \sqrt{k}$,	\\
\frac{e^{-\frac{t^2}{k}}}{2^k \sqrt{k}} & if $\sqrt{k} \leq |t| \leq k$,	\\
\frac{k^{\frac{3}{2}} e^{-|t|}}{2^k |t|^2} & if $|t| \geq k$.
\end{dcases*}\]
\end{lemma}

\begin{proof}
By Stirling's approximation \eqref{eqn:Stirling},
\begin{multline*}
\log |h(t,k)| = \left(k - \frac{1}{2}\right) \log k - 2(k - 1) \log \left(k - \frac{1}{2}\right) + (k - 1) \log \left(\frac{k^2}{4} + t^2\right)	\\
- \frac{k}{2} \log \left(\frac{(k + 1)^2}{4} + t^2\right) - 4|t| \arctan \frac{2|t|}{k} + 2|t| \arctan \frac{2|t|}{k + 1} + O(1).
\end{multline*}
Via the arctangent subtraction formula $\arctan x - \arctan y = \arctan \frac{x - y}{1 + xy}$ and the fact that $\arctan x = O(|x|)$ and $\log(1 + x) = O(1/|x|)$, we deduce that
\[\log |h(t,k)| = \log \frac{k^{\frac{3}{2}}}{k^2 + 4t^2} - k \log 2 + \frac{k}{2} \log \left(1 + \frac{4t^2}{k^2}\right) - 2|t| \arctan \frac{2|t|}{k} + O(1).\]
The desired bound then follows from the inequalities
\[\frac{k}{2} \log \left(1 + \frac{4t^2}{k^2}\right) - 2|t| \arctan \frac{2|t|}{k} \leq \begin{dcases*}
O(1) & if $|t| \leq \sqrt{k}$,	\\
-\frac{t^2}{k} & if $\sqrt{k} \leq |t| \leq k$,	\\
-|t| & if $|t| > k$.
\end{dcases*}\qedhere\]
\end{proof}

\begin{lemma}
\label{lem:tildehtkbound}
For $t \in \R$, we have that
\[\widetilde{h}(t,k) \ll \begin{dcases*}
\frac{\sqrt{k} e^{-\frac{\pi}{2}|t|}}{1 + |t|} & if $|t| \leq k$,	\\
\frac{k e^{-\frac{\pi}{2}|t|}}{|t|^{\frac{3}{2}}} & if $|t| \geq k$.
\end{dcases*}\]
\end{lemma}

\begin{proof}
By Stirling's approximation \eqref{eqn:Stirling},
\begin{multline*}
\log |\widetilde{h}(t,k)| = \left(k - \frac{1}{2}\right) \log k - 2(k - 1) \log \left(k - \frac{1}{2}\right) + \frac{1}{2} (k - 1) \log \left(\left(k - \frac{1}{2}\right)^2 + t^2\right)	\\
- |t| \arctan \frac{|t|}{k - \frac{1}{2}} - \frac{1}{2} \log \left(\frac{1}{4} + t^2\right) - |t| \arctan 2|t| + O(1),
\end{multline*}
which simplifies to
\begin{multline*}
\log |\widetilde{h}(t,k)| = \log \frac{k}{(1 + 4t^2)^{\frac{1}{2}} \left(\left(k - \frac{1}{2}\right)^2 + t^2\right)^{\frac{1}{4}}} - |t| \arctan 2|t|	\\
+ \frac{1}{2} \left(k - \frac{1}{2}\right) \log \left(1 + \frac{t^2}{\left(k - \frac{1}{2}\right)^2}\right) - |t| \arctan \frac{|t|}{k - \frac{1}{2}} + O(1).
\end{multline*}
Since $\arctan 2|t| = \frac{\pi}{2} + O(1)$ and
\[\frac{1}{2} \left(k - \frac{1}{2}\right) \log \left(1 + \frac{t^2}{\left(k - \frac{1}{2}\right)^2}\right) - |t| \arctan 2|t| \leq \begin{dcases*}
O(1) & if $|t| \leq \sqrt{k}$,	\\
-\frac{t^2}{k} & if $\sqrt{k} \leq |t| \leq k$,	\\
-|t| & if $|t| \geq k$,
\end{dcases*}\]
the result follows.
\end{proof}

\subsection{The Size of Moments of \texorpdfstring{$L$}{L}-Functions}

We now use \hyperref[lem:hholellkbound]{Lemma \ref*{lem:hholellkbound}} to extract the second moment in the conductor-dropping range
\[\sum_{\substack{f \in \BB_{\hol} \\ k_f = k}} \frac{L\left(\frac{1}{2},f \otimes g\right)^2}{L(1,\ad f)}\]
from the moment $\MM_{\hol}$ given by \eqref{eqn:MMholdefeq}.

\begin{proposition}
\label{prop:MMholbounds}
We have that
\[\MM_{\hol} = \sum_{\substack{f \in \BB_{\hol} \\ k_f = k}} \frac{L\left(\frac{1}{2},f \otimes g\right)^2}{L(1,\ad f)} + O_{\e}(k^{\e}).\]
\end{proposition}

\begin{proof}
From \hyperref[lem:hholellkbound]{Lemma \ref*{lem:hholellkbound}} and the fact that $h^{\hol}(k,k) = 1$, we have that
\[\MM_{\hol} - \sum_{\substack{f \in \BB_{\hol} \\ k_f = k}} \frac{L\left(\frac{1}{2},f \otimes g\right)^2}{L(1,\ad f)} \ll \sum_{\substack{\ell = 2 \\ \ell \equiv 0 \hspace{-.25cm} \pmod{2}}}^{k - 2} \frac{1}{k - \ell} \left(\frac{k - \ell}{2(k - 1)}\right)^{\frac{k - \ell}{2}} \sum_{\substack{f \in \BB_{\hol} \\ k_f = \ell}} \frac{L\left(\frac{1}{2},f \otimes g\right)^2}{L(1,\ad f)}.\]
Since the analytic conductor of $L(\frac{1}{2},f \otimes g)$ for $f \in \BB_{\hol}$ with $k_f \leq k$ is $\asymp k^2 (1 + k - k_f)^2$, the spectral large sieve implies that for $\ell < k$,
\[\sum_{\substack{f \in \BB_{\hol} \\ k_f = \ell}} \frac{L\left(\frac{1}{2},f \otimes g\right)^2}{L(1,\ad f)} \ll_{\e} k^{1 + \e} (k - \ell).\]
It remains to note that
\[\sum_{\substack{\ell = 2 \\ \ell \equiv 0 \hspace{-.25cm} \pmod{2}}}^{k - 2} \left(\frac{k - \ell}{2(k - 1)}\right)^{\frac{k - \ell}{2}} \ll \frac{1}{k}.\qedhere\]
\end{proof}

We then use \hyperref[lem:htkbound]{Lemma \ref*{lem:htkbound}} to bound the moments $\MM_{\Maass}$ and $\MM_{\Eis}$ given by \eqref{eqn:MMMaassdefeq} and \eqref{eqn:MMEisdefeq} respectively.

\begin{proposition}
\label{prop:MMMaassEisbounds}
We have that
\[\begin{drcases*}
\MM_{\Maass} \\
\MM_{\Eis}
\end{drcases*} \ll_{\e} k^{\frac{3}{2} + \e} 2^{-k}.\]
\end{proposition}

\begin{proof}
By \hyperref[lem:htkbound]{Lemma \ref*{lem:htkbound}}, we have that
\begin{multline}
\label{eqn:MMMaassbounds}
\MM_{\Maass} \ll \frac{1}{2^k \sqrt{k}} \sum_{\substack{f \in \BB_0 \\ t_f \leq \sqrt{k}}} \frac{L\left(\frac{1}{2},f \otimes g\right)^2}{L(1,\ad f)} + \frac{1}{2^k \sqrt{k}} \sum_{\substack{f \in \BB_0 \\ \sqrt{k} \leq t_f \leq k}} \frac{L\left(\frac{1}{2},f \otimes g\right)^2}{L(1,\ad f)} e^{-\frac{t_f^2}{k}}	\\
+ \frac{k^{\frac{3}{2}}}{2^k} \sum_{\substack{f \in \BB_0 \\ \sqrt{k} \leq t_f \leq k}} \frac{L\left(\frac{1}{2},f \otimes g\right)^2}{L(1,\ad f)} \frac{e^{-t_f}}{t_f^2}.
\end{multline}
Since the analytic conductor of $L(\frac{1}{2},f \otimes g)$ for $f \in \BB_0$ is $\asymp k^4$ if $t_f \leq k$ and is $\asymp t_f^4$ if $t_f \geq k$, the spectral large sieve implies that
\[\sum_{\substack{f \in \BB_0 \\ T \leq t_f \leq 2T}} \frac{L\left(\frac{1}{2},f \otimes g\right)^2}{L(1,\ad f)} \ll_{\e} \begin{dcases*}
k^{2 + \e} & if $T \leq k$,	\\
T^{2 + \e} & if $T \geq k$.
\end{dcases*}\]
By breaking up the sums over $f \in \BB_0$ in \eqref{eqn:MMMaassbounds} into dyadic intervals, we deduce the desired bounds for $\MM_{\Maass}$. An analogous argument yields the desired bounds for $\MM_{\Eis}$.
\end{proof}

Finally, we use \hyperref[lem:tildehtkbound]{Lemma \ref*{lem:tildehtkbound}} to bound the moments $\widetilde{\MM}_{\Maass}$ and $\widetilde{\MM}_{\Eis}$ given by \eqref{eqn:tildeMMMaassdefeq} and \eqref{eqn:tildeMMEisdefeq} respectively.

\begin{proposition}
\label{prop:tildeMMMaassEisbounds}
We have that
\[\begin{drcases*}
\widetilde{\MM}_{\Maass}	\\
\widetilde{\MM}_{\Eis}
\end{drcases*} \ll_{\e} k (\log k)^{-\frac{1}{2} + \e}.\]
\end{proposition}

\begin{proof}
From \hyperref[lem:tildehtkbound]{Lemma \ref*{lem:tildehtkbound}}, we have that
\[\widetilde{\MM}_{\Maass} \ll \sqrt{k} \sum_{\substack{f \in \BB_0 \\ t_f \leq k}} \frac{\sqrt{L\left(\frac{1}{2},\ad g \otimes f\right) L\left(\frac{1}{2},f\right)^5}}{L(1,\ad f)} \frac{e^{-\frac{\pi}{2} t_f}}{1 + t_f} + k \sum_{\substack{f \in \BB_0 \\ t_f \geq k}} \frac{\sqrt{L\left(\frac{1}{2},\ad g \otimes f\right) L\left(\frac{1}{2},f\right)^5}}{L(1,\ad f)} \frac{e^{-\frac{\pi}{2} t_f}}{t_f^{\frac{3}{2}}}.\]
We break up the sums over $f \in \BB_0$ above into intervals of length one. Since for $f \in \BB_0$, the analytic conductor of $L(\frac{1}{2},\ad g \otimes f)$ is $(k + t_f)^4 (1 + t_f)^2$, we may bound $L(\frac{1}{2},\ad g \otimes f)$ pointwise via the weak subconvexity bound \cite[Theorem 1.3]{Watt24}
\[L\left(\frac{1}{2},\ad g \otimes f\right) \ll_{\e} \begin{dcases*}
k (\log k)^{-1 + \e} (1 + t_f)^{\frac{1}{2} + \e} & if $t_f \leq k$,	\\
t_f^{\frac{3}{2} + \e} & if $t_f \geq k$.
\end{dcases*}\]
We then bound the remaining fractional moment of $L$-functions by applying the Cauchy--Schwarz inequality and invoking the second- and third-moment bounds
\begin{align*}
\sum_{\substack{f \in \BB_0 \\ T - 1 \leq t_f \leq T}} \frac{L\left(\frac{1}{2},f\right)^2}{L(1,\ad f)} & \ll_{\e} T^{1 + \e},	\\
\sum_{\substack{f \in \BB_0 \\ T - 1 \leq t_f \leq T}} \frac{L\left(\frac{1}{2},f\right)^3}{L(1,\ad f)} & \ll_{\e} T^{1 + \e}
\end{align*}
for $T \geq 1$; here the former is a consequence of the spectral large sieve, while the latter is a result of Ivi\'{c} \cite[Theorem]{Ivi01}. In this way, we obtain the desired bound for $\widetilde{\MM}_{\Maass}$. An analogous argument yields the desired bounds for $\widetilde{\MM}_{\Eis}$.
\end{proof}

\section{Bounds for the Main Term}

The remaining step required to complete the proof of \hyperref[thm:mainthm]{Theorem \ref*{thm:mainthm}} is to determine the size of the main term $\widetilde{\MM}_0$ given by \eqref{eqn:tildeMM0defeq}.

\begin{proposition}
\label{prop:tildeMM0bounds}
\hspace{1em}

\begin{enumerate}[leftmargin=*,label=\textup{(\arabic*)}]
\item\label{propno:tildeMM0bounds1} For $j \in \{0,1,2,3\}$, there exist polynomials $P_j$ of degree $j$ such that
\begin{equation}
\label{eqn:M0-expression}
\widetilde{\MM}_0 = k \sum_{j = 0}^{3} P_{j}(\log k) L^{(3 - j)}(1,\ad g) + O(\sqrt{k}).
\end{equation}
\item\label{propno:tildeMM0bounds2} We have the upper bound
\begin{equation}
\label{eqn:M0-upper-bound}
\widetilde{\MM}_0 \ll k (\log k)^6.
\end{equation}
\item\label{propno:tildeMM0bounds3} The main term $\widetilde{\MM}_0$ is positive for $k$ sufficiently large and satisfies the lower bound
\begin{equation}
\label{eqn:M0-lower-bound}
\widetilde{\MM}_0 \gg k (\log k)^2.
\end{equation}
\end{enumerate}
\end{proposition}

\begin{remark}
Conditionally on the Riemann hypothesis for $L(s,\ad g)$, the dominating term in \eqref{eqn:M0-expression} is $k (\log k)^3 L(1,\ad g)$. Inserting in this term the unconditional lower bound
\[L(1,\ad g) \gg \frac{1}{\log k}\]
due to Goldfeld, Hoffstein, and Lieman \cite{GHL94} gives \eqref{eqn:M0-lower-bound}. Thus \eqref{eqn:M0-lower-bound} is the best lower bound one can realistically hope to achieve unconditionally for $\widetilde{\MM}_0$.
\end{remark}

The proof of \hyperref[prop:tildeMM0bounds]{Proposition \ref*{prop:tildeMM0bounds}} requires various bounds for $L(s,\ad g)$ and its derivatives.

\begin{lemma}
\label{lem:bounds}
For $|t| \ll k$ and $j \geq 0$, we have that
\begin{align*}
L^{(j)}(1,\ad g) & \ll_j (\log k)^{3 + j},	\\
L^{(j)}\left(\frac{1}{2} + it,\ad g\right) & \ll_{j,\e} \sqrt{k} (\log k)^{j - 1 + \e} (1 + |t|)^{\frac{3}{4}}.
\end{align*}
\end{lemma}

\begin{proof}
The first upper bound follows by writing $L^{(j)}(1,\ad g)$ as an essentially finite Dirichlet series, and then bounding the sum absolutely (see \cite[Proposition 3.2 (i)]{LW06}). The bound on $L^{(j)}(\frac{1}{2} + it,\ad g)$ is an instance of Soundararajan's weak subconvexity. Writing $L(\frac{1}{2} + it,\ad g)$, which has analytic conductor $(k + |t|)^2 (1 + |t|)$, in terms of an approximate functional equation, differentiating this $j$ times, and then applying \cite[Theorem 2]{Sou10} together with partial summation (as in \cite[Page 1486]{Sou10}) gives the claim.
\end{proof}

Our starting point for \hyperref[prop:tildeMM0bounds]{Proposition \ref*{prop:tildeMM0bounds}} is the identity $\widetilde{\MM}_0 = \lim_{(w_1,w_2) \to (0,0)} \widetilde{\MM}_0(w_1,w_2)$, where $\widetilde{\MM}_0(w_1,w_2)$ is the function given by \eqref{eqn:tildeMM0w1w2defeq} in terms of the related function $\Psi(w_1,w_2)$ given by \eqref{eqn:Psiw1w2defeq}.

\begin{lemma}
We have that
\begin{equation}
\label{eqn:M0res}
\widetilde{\MM}_0 = 4 \Res_{w_2 = 0} \left( \Res_{w_1 = 0} \frac{\Psi(w_1,w_2)}{w_1w_2}\right).
\end{equation}
\end{lemma}

\begin{proof}
Let $\Re(w_2) > 0$. Taking the Laurent expansion of $\Psi(w_1, w_2)$ about $w_1 = 0$, where there is a simple pole, we can write
\[\Psi(\pm_1 w_1, \pm_2 w_2) = \frac{a_{-1}(\pm_2 w_2)}{\pm_1 w_1} + a_{0}(\pm_2 w_2) \pm_1 w_1 a_{1}(\pm_2 w_2) + \ldots,\]
where $a_j$ are meromorphic functions for $j \geq -1$. From the proof of \hyperref[prop:secondexpansion]{Proposition \ref*{prop:secondexpansion}}, $\widetilde{\MM}_0(w_1,w_2)$ is holomorphic at $(0,w_2)$ for $-\frac{1}{2} < \Re(w_2) < \frac{1}{2}$, and so we must have that
\[\sum_{\pm_1,\pm_2} \frac{a_{-1}(\pm_2 w_2)}{\pm_1 w_1} = 0\]
as a function of $w_1$ and $w_2$. Thus, inserting the value $w_1 = 0$ into the other terms, we have that
\[\widetilde{\MM}_0(0,w_2) = \sum_{\pm_1,\pm_2} a_{0}(\pm_2 w_2) = 2\sum_{\pm_2} a_{0}(\pm_2 w_2).\]
We can next take the Laurent expansion of $a_{0}(w_2)$ about $w_2 = 0$, where there is a triple pole, argue as before to insert the value $w_2 = 0$, and obtain
\[\widetilde{\MM}_0(0,w_2) = 2 \sum_{\pm_2} \left(\frac{b_{-3}}{(\pm_2 w_2)^3} + \frac{b_{-2}}{(\pm_2 w_2)^2} + \frac{b_{-1}}{\pm_1 w_2} + b_{0} \pm b_1 w_2 + \ldots\right) = 2 \sum_{\pm_2} b_{0} = 4 b_{0}.\]
The constant term in the Laurent expansion about $z = 0$ of a meromorphic function $f(z)$ can be picked out as the residue at $z = 0$ of $\frac{f(z)}{z}$ by Cauchy's residue theorem. With this fact and since $b_0$ is the constant term in the Laurent expansion in $w_2$ of the constant term of the Laurent expansion in $w_1$ of $\Psi(w_1, w_2)$, we deduce the lemma.
\end{proof}

Next, we slightly modify the form of $\Psi(w_1,w_2)$ in a way that does not affect our asymptotic evaluation of $\widetilde{\MM}_0$. Defining
\begin{equation}
\label{eqn:psiw1w2defeq}
\psi(w_1,w_2) \coloneqq \frac{\zeta(1 + w_1 + w_2) L(1 + w_1 + w_2, \ad g)}{\zeta(2 + 2w_1 + 2w_2)} \prod_{j = 1}^{2} \left(\frac{k}{4\pi^2}\right)^{w_j} \Gamma\left(\frac{1}{2} + w_j\right) \zeta(1 + 2w_j) (1 - 64 w_j^6),
\end{equation}
we have the following.

\begin{corollary}
We have that
\begin{equation}
\label{eqn:new-psi}
\widetilde{\MM}_0 = 8 k \Res_{w_2 = 0} \left(\Res_{w_1 = 0} \frac{\psi(w_1,w_2)}{w_1 w_2}\right) + O(\sqrt{k}).
\end{equation}
\end{corollary}

Note that by having multiplied $\Psi(w_1,w_2)$ by $(1 - 64 w_1^6)(1 - 64w_2^6)$, we have cancelled the poles of $\Gamma(\frac{1}{2} + w_1) \Gamma(\frac{1}{2} + w_2)$ at $w_1 = -\frac{1}{2}$ and $w_2 = -\frac{1}{2}$, which shall be convenient later.

\begin{proof}
Multiplying $\Psi(w_1,w_2)$ by $(1 - 64 w_1^6)(1 - 64w_2^6)$ does not alter the residue calculation in \eqref{eqn:M0res} because the extra factor $(1 - 64 w_1^6) (1 - 64 w_2^6)$ is equal to $1$ at $(w_1,w_2) = (0,0)$ and all of its partial derivatives up to the fifth order vanish at $(w_1,w_2) = (0,0)$. Next, by applying Stirling's approximation \eqref{eqn:Stirling}, we have that
\begin{equation}
\label{eqn:psi-stirling}
\frac{\Gamma(k)^2}{\Gamma\left(k - \frac{1}{2}\right)^2} \frac{\Gamma(w_1 + w_2 + k)}{\Gamma(k)} = k^{1 + w_1 + w_2} + O\left(k^{\Re(w_1) + \Re(w_2)}\right)
\end{equation}
for $|w_1|, |w_2| < 1$. To see that the error term in \eqref{eqn:psi-stirling} contributes $O(\sqrt{k})$ to the residue calculation, as seen in \eqref{eqn:new-psi}, one can pick out the residues using Cauchy's integral theorem, integrating over small circles about $w_1 = 0$ and $w_2 = 0$, and bounding the contribution of the error term in \eqref{eqn:psi-stirling} around these circles.
\end{proof}

\begin{proof}[Proof of {\hyperref[propno:tildeMM0bounds1]{Proposition \ref*{prop:tildeMM0bounds} (1)}} and {\hyperref[propno:tildeMM0bounds2]{(2)}}]
For $\Re(w_2) \geq -\frac{1}{2}$ with $w_2 \neq 0$, the function $\frac{\psi(w_1,w_2)}{w_1 w_2}$ has a double polar line at $w_1 = 0$, and we compute that
\begin{multline}
\label{eqn:residue}
\Res_{w_1 = 0} \frac{\psi(w_1,w_2)}{w_1 w_2} \\
= \sum_{\substack{j_1, j_2, j_3 \in \{0,1\} \\ j_1 + j_2 + j_3 = 1}} \frac{\zeta(1 + 2w_2)\zeta^{(j_1)}(1 + w_2)}{w_2} L^{(j_2)}(1 + w_2,\ad g) \left(\log \frac{k}{4\pi^2}\right)^{j_3} \left(\frac{k}{4\pi^2}\right)^{w_2} h_{j_1,j_2,j_3}(w_2)
\end{multline}
for some functions $h_{j_1,j_2,j_3}(w_2)$ that are holomorphic for $\Re(w_2) \geq -\frac{1}{2}$. The function \eqref{eqn:residue} has a pole of order $4$ at $w_2 = 0$, and from this \eqref{eqn:M0-expression} follows. The upper bound \eqref{eqn:M0-upper-bound} then follows from \eqref{eqn:M0-expression} by invoking the upper bounds $L^{(3 - j)}(1,\ad g) \ll (\log k)^{6 - j}$ from \hyperref[lem:bounds]{Lemma \ref*{lem:bounds}}.
\end{proof}

To establish \hyperref[propno:tildeMM0bounds3]{Proposition \ref*{prop:tildeMM0bounds} (3)}, we need to express the desired residues in a different way.

\begin{lemma}
Let
\begin{equation}
\label{eqn:constantdefeq}
c \coloneqq \frac{1}{2\pi i} \int_{\frac{1}{4} - i\infty}^{\frac{1}{4} + i\infty} \frac{(1 - 64w^2)^2}{w^2} \prod_{\pm} \Gamma\left(\frac{1}{2} \pm w\right) \zeta(1 \pm 2w) \, dw.
\end{equation}
We have that
\[\widetilde{\MM}_0 = 8k \left(\frac{1}{(2\pi i)^2} \int_{\frac{1}{4} - i\infty}^{\frac{1}{4} + i\infty} \int_{\frac{1}{4} - i\infty}^{\frac{1}{4} + i\infty} \frac{\psi(w_1,w_2)}{w_1 w_2} \, dw_1 dw_2 + c \frac{L(1,\ad g)}{\zeta(2)} \right) + O_{\e}(k (\log k)^{\e}).\]
\end{lemma}

\begin{proof}
We consider the double integral
\[\frac{1}{(2\pi i)^2} \int_{\frac{1}{4} - i\infty}^{\frac{1}{4} + i\infty} \int_{\frac{1}{4} - i\infty}^{\frac{1}{4} + i\infty} \psi(w_1,w_2) \, \frac{dw_1 \, dw_2}{w_1 w_2}.\]
This double integral converges absolutely by the exponential decay of $\Gamma(\frac{1}{2} + w)$ as $|\Im(w)|$ tends to infinity via Stirling's approximation \eqref{eqn:Stirling}. We move the $w_1$-integral left to the line $\Re(w_1) = -\frac{3}{4}$, crossing a double pole at $w_1 = 0$ and a simple pole at $w_1 = -w_2$. By Cauchy's residue theorem, we find that this double integral is equal to
\begin{equation}
\label{eqn:M0cauchy}
\frac{1}{(2\pi i)^2} \int_{\frac{1}{4} - i\infty}^{\frac{1}{4} + i\infty} \int_{-\frac{3}{4} - i\infty}^{-\frac{3}{4} + i\infty} \frac{\psi(w_1,w_2)}{w_1 w_2} \, dw_1 \, dw_2 + \frac{1}{2\pi i} \int_{\frac{1}{4} - i\infty}^{\frac{1}{4} + i\infty} \left(\Res_{w_1 = 0} \frac{\psi(w_1,w_2)}{w_1 w_2}\right) \, dw_2 - c \frac{L(1,\ad g)}{\zeta(2)},
\end{equation}
where $c$ is as in \eqref{eqn:constantdefeq}. The first term on the right hand side of \eqref{eqn:M0cauchy} is $\ll_{\e} (\log k)^{-1 + \e}$ since $\Re(w_1) + \Re(w_2) = -\frac{1}{2}$ on these lines and 
\[k^{w_1 + w_2} L(1 + w_1 + w_2,\ad g) \ll_{\e} (\log k)^{-1 + \e} (1 + |\Im(w_1) + \Im(w_2)|)^{\frac{3}{4}}\]
by \hyperref[lem:bounds]{Lemma \ref*{lem:bounds}}. The second term on the right hand side of \eqref{eqn:M0cauchy} equals, after shifting the contour left to the line $\Re(w_2) = -\frac{1}{2}$ and crossing a pole at $w_2 = 0$,
\[\Res_{w_2 = 0} \left(\Res_{w_1 = 0} \frac{\psi(w_1,w_2)}{w_1 w_2} \right) + \frac{1}{2\pi i} \int_{-\frac{1}{2} - i\infty}^{-\frac{1}{2} + i\infty} \left(\Res_{w_1 = 0} \frac{\psi(w_1,w_2)}{w_1 w_2} \right) \, dw_2.\]
The second term above is $\ll_{\e} (\log k)^{\e}$ using the identity \eqref{eqn:residue} and \hyperref[lem:bounds]{Lemma \ref*{lem:bounds}}. Upon recalling the identity \eqref{eqn:new-psi} relating the first term above to $\widetilde{\MM}_0$, the lemma is complete.
\end{proof}

\hyperref[propno:tildeMM0bounds3]{Proposition \ref*{prop:tildeMM0bounds} (3)} then follows by establishing the following.

\begin{lemma}
We have that
\begin{equation}
\label{eqn:cLlowerbound}
c \frac{L(1,\ad g)}{\zeta(2)} + \frac{1}{(2\pi i)^2} \int_{\frac{1}{4} - i\infty}^{\frac{1}{4} + i\infty} \int_{\frac{1}{4} - i\infty}^{\frac{1}{4} + i\infty} \frac{\psi(w_1,w_2)}{w_1 w_2} \, dw_1 \, dw_2 \gg (\log k)^2.
\end{equation}
\end{lemma}

\begin{proof}
We consider the integral
\begin{equation}
\label{eqn:int14}
\frac{c}{2\pi i} \int_{\frac{1}{4} - i\infty}^{\frac{1}{4} + i\infty} \frac{\zeta(1 + w) L(1 + w,\ad g)}{\zeta(2 + 2w)} \left(\frac{k}{(\log k)^2}\right)^{w} e^{w^2} \, dw.
\end{equation}
The factor $e^{w^2}$ ensures absolute convergence of the integral.

On the one hand, we may express the ratio of $L$-functions as an absolutely convergent Dirichlet series, namely
\[\frac{\zeta(1 + w) L(1 + w,\ad g)}{\zeta(2 + 2w)} = \sum_{n = 1}^{\infty} \frac{\lambda_g(n)^2}{n^{1 + w}},\]
and exchange the order of summation and integration in order to write the integral \eqref{eqn:int14} as
\[\sum_{n = 1}^{\infty} \frac{\lambda_g(n)^2}{n} W_1(n),\]
where
\begin{equation}
\label{eqn:W1def}
W_1(n) \coloneqq \frac{c}{2\pi i} \int_{\frac{1}{4} - i\infty}^{\frac{1}{4} + i\infty} \left(\frac{k}{n (\log k)^2}\right)^{w} e^{w^2} \, dw.
\end{equation}
On the other hand, we may move the integral left to the line $\Re(w) = -\frac{1}{2}$, crossing a simple pole at $w = 0$. By Cauchy's residue theorem, we find that the integral \eqref{eqn:int14} is equal to
\[c \frac{L(1,\ad g)}{\zeta(2)} + \frac{c}{2\pi i } \int_{-\frac{1}{2} - i\infty}^{-\frac{1}{2} + i\infty} \frac{\zeta(1 + w) L(1 + w,\ad g)}{\zeta(2 + 2w)} \left(\frac{k}{(\log k)^2}\right)^{w} e^{w^2} \, dw.\]
The second term above is $\ll_{\e} (\log k)^{\e}$ using \hyperref[lem:bounds]{Lemma \ref*{lem:bounds}}. Combining these two expressions for the integral \eqref{eqn:int14}, we find that
\[c \frac{L(1,\ad g)}{\zeta(2)} = \sum_{n = 1}^{\infty} \frac{\lambda_g(n)^2}{n} W_1(n) + O_{\e}((\log k)^{\e}).\]

Similarly, we express the ratio of $L$-functions appearing in the definition \eqref{eqn:psiw1w2defeq} of $\psi(w_1,w_2)$ as an absolutely convergent Dirichlet series and exchange the order of summation and integration in order to obtain the identity
\[\frac{1}{(2\pi i)^2} \int_{\frac{1}{4} - i\infty}^{\frac{1}{4} + i\infty} \int_{\frac{1}{4} - i\infty}^{\frac{1}{4} + i\infty} \frac{\psi(w_1,w_2)}{w_1 w_2} \, dw_1 \, dw_2 = \sum_{n = 1}^{\infty} \frac{\lambda_g(n)^2}{n} W_2(n),\]
where
\begin{equation}
\label{eqn:W2def}
W_2(n) \coloneqq \left(\frac{1}{2\pi i} \int_{\frac{1}{4} - i\infty}^{\frac{1}{4} + i\infty} \left(\frac{k}{4\pi^2 n}\right)^{w} \Gamma\left(\frac{1}{2} + w\right) \zeta(1 + 2w) (1 - 64 w^6) \, \frac{dw}{w} \right)^2.
\end{equation}

It follows that the left-hand side of \eqref{eqn:cLlowerbound} is equal to
\[\sum_{n = 1}^{\infty} \frac{\lambda_g(n)^2}{n} (W_1(n) + W_2(n)) + O_{\e}((\log k)^{\e}).\]
It therefore suffices to show that
\[\sum_{n = 1}^{\infty} \frac{\lambda_g(n)^2}{n} (W_1(n) + W_2(n)) \gg (\log k)^2.\]
We break the sum into four parts and show the following, for $k$ sufficiently large:
\begin{align}
\label{eqn:toshow1}
\lambda_g(1)^2 (W_1(1) + W_2(1)) & \gg (\log k)^2,	\\
\label{eqn:toshow2}
\sum_{2 \leq n \leq \frac{k}{\log k}} \frac{\lambda_g(n)^2}{n} (W_1(n) + W_2(n)) & \geq 0,	\\
\label{eqn:toshow3}
\sum_{\frac{k}{\log k} < n \leq k^2} \frac{\lambda_g(n)^2}{n} (W_1(n) + W_2(n)) & \geq -C(\log k)^{-10} \quad \text{for some constant $C \geq 0$},	\\
\label{eqn:toshow4}
\sum_{n > k^2} \frac{\lambda_g(n)^2}{n} (W_1(n) + W_2(n)) & \ll k^{-10}.
\end{align}

To begin, we observe that
\begin{align}
\label{eqn:W1bounded}
W_1(n) & \ll 1 \quad \text{for all $n \in \N$},	\\
\label{eqn:W1milddecay}
W_1(n) & \ll (\log k)^{-20} \quad \text{for $\frac{k}{\log k} < n \leq k^2$},	\\
\label{eqn:W1decay}
W_1(n) & \ll (nk)^{-10} \quad \text{for $n > k^2$},	\\
\label{eqn:W2pos}
W_2(n) & \geq 0 \quad \text{for all $n \in \N$},	\\
\label{eqn:W2asymp}
W_2(n) & = \left(\log \frac{k}{4\pi^2 n} + O(1)\right)^2 \quad \text{for $n < \frac{k}{4\pi^2}$},	\\
\label{eqn:W2decay}
W_2(n) & \ll (nk)^{-10} \quad \text{for $n > k^2$}.
\end{align}
Indeed, the bound \eqref{eqn:W1bounded} for $W_1(n)$ follows by moving the line of integration in \eqref{eqn:W1def} to the left of $\Re(w) = 0$ for $n < \frac{k}{(\log k)^2}$ and staying to the right of $\Re(w)=0$ for $n \geq \frac{k}{(\log k)^2}$. The bound \eqref{eqn:W1milddecay} follows by moving to the line $\Re(w) = 20$ and using the fact $\frac{k}{n (\log k)^2} < (\log k)^{-1}$ for $\frac{k}{\log k} < n \leq k^2$, while the bound \eqref{eqn:W1decay} follows by moving far to the right for $n > k^2$. The nonnegativity \eqref{eqn:W2pos} of $W_2(n)$ is immediate from the definition \eqref{eqn:W2def}. Finally, the asymptotic formula \eqref{eqn:W2asymp} for $W_2(n)$ for $n < \frac{k}{4\pi^2}$ follows by moving the line of integration in \eqref{eqn:W2def} far to the left and crossing a double pole at $w = 0$, while the upper bound \eqref{eqn:W2decay} for $W_2(n)$ for $n > k^2$ follows by instead moving the line of integration far to the right.

With these estimates for $W_1(n)$ and $W_2(n)$ in hand, we may now complete the proof.
\begin{enumerate}
\item For the first claim \eqref{eqn:toshow1}, observe that $\lambda_g(1) = 1$, $W_1(1) \ll 1$ from \eqref{eqn:W1bounded}, and $W_2(1) \gg (\log k)^2$ from \eqref{eqn:W2asymp}.
\item For the second claim, observe that for $2 \leq n \leq \frac{k}{\log k}$, we have that $W_1(n) \ll 1$ by \eqref{eqn:W1bounded} and $W_2(n) \gg (\log \log k)^2$ by \eqref{eqn:W2asymp} and the fact that $\frac{k}{n} \geq \log k$. These facts together with the nonnegativity of $\frac{\lambda_g(n)^2}{n}$ give \eqref{eqn:toshow2}.
\item For the third claim, by inserting the bound \eqref{eqn:W1milddecay} for $W_1(n)$ and the bound
\[\sum_{\frac{k}{\log k} < n \leq k^2} \frac{\lambda_g(n)^2}{n} \leq \sum_{n \leq k^2} \frac{d(n)^2}{n} \ll (\log k)^3,\]
we get
\[\sum_{\frac{k}{\log k} < n \leq k^2} \frac{\lambda_g(n)^2}{n} W_1(n) \ll (\log k)^{-10}.\]
Since
\[\sum_{\frac{k}{\log k} < n \leq k^2} \frac{\lambda_g(n)^2}{n} W_2(n) \geq 0\]
by nonnegativity, we deduce \eqref{eqn:toshow3}.
\item Finally, the fourth claim \eqref{eqn:toshow4} follows immediately from the fact that $W_1(n) \ll (nk)^{-10}$ and $W_2(n) \ll (nk)^{-10}$ for $n > k^2$ via \eqref{eqn:W1decay} and \eqref{eqn:W2decay}.
\qedhere
\end{enumerate}
\end{proof}

\end{document}